\documentclass[11pt,leqno]{article}
\usepackage{amsmath,amsthm}
\usepackage{amsfonts}
\usepackage{mathrsfs}
\usepackage{hyperref}
\usepackage{textcomp}
\usepackage{amssymb}
\usepackage{amsbsy}
\usepackage{epsfig}
\usepackage{graphicx}
\usepackage{caption}
\usepackage{bm}
\usepackage{color} 

 \usepackage[left=1.3in, right=1in, top=1.5in, bottom=1.3in]{geometry}
\numberwithin{equation}{section} 
\newtheorem{theorem}{\bf Theorem}[section]
\newtheorem{example}{\bf Example}[section]

\newtheorem{remark}{\bf Remark}[section]
\newtheorem{lemma}{\bf Lemma}[section]

\newcommand{\bphi}{\mbox{\boldmath $\phi$}}

 \newcommand{\norm}[1]{\left\lVert #1\right\rVert}
\newcommand{\tnorm}[1]{\ensuremath{\left| \! \left| \! \left|} #1 \ensuremath{\right| \! \right| \! \right|}}
\newsavebox{\savepar}		 

\begin{document}
\title{Global Stabilization of BBM-Burgers' Type Equations by Nonlinear Boundary Feedback Control Laws: Theory
 and Finite Element Error Analysis}
\author{ Sudeep Kundu\footnote{
		Institute of Mathematics and Scientific Computing, University of Graz,
	Heinrichstr. 36, A-8010 Graz, Austria,
	Email: sudeep.kundu@uni-graz.at} 
	and
	Amiya Kumar Pani\footnote{
	Department of Mathematics, 	IIT Bombay, 
  Powai, Mumbai-400076, India,
	Email:{akp@math.iitb.ac.in}}.}
\maketitle
\abstract{ In this article, global stabilization results for the Benjamin-Bona-Mahony-Burgers' (BBM-B) type 
equations are obtained using nonlinear Neumann boundary feedback control laws.
Based on the $C^0$-conforming finite element method, global stabilization results for the semidiscrete 
solution are also discussed. Optimal error estimates in $L^\infty(L^2)$, $L^\infty(H^1)$ 
and $L^\infty(L^\infty)$-norms for the state variable are derived, which preserve exponential stabilization 
property. Moreover, for the first time in the literature, superconvergence results  for the boundary 
feedback control laws are established.
Finally, several numerical experiments are conducted to confirm our theoretical findings.}

\section{Introduction}
Consider the Benjamin-Bona-Mahony-Burgers' (BBM-B) equations of the following type: seek $u=u(x,t),\hspace{0.1cm} x\in I=(0,1)$ and $t>0$ which satisfies
\begin{align}
	&u_t-\mu u_{xxt}-\nu u_{xx}+u_x+uu_x=0,\qquad  (x,t)\in(0,1)\times(0,\infty)\label{deq1.1},\\
	&u_x(0,t)=v_0(t),   \qquad t\in (0,\infty)\label{deq1.2},\\
	&u_x(1,t)=v_1(t), \qquad t\in (0,\infty)\label{deq1.3},\\
	&u(x,0)=u_0(x),\qquad x\in (0,1)\label{deq1.4},
\end{align}
where, the dispersion coefficient $\mu>0$  and the dissipative coefficient $\nu>0$  are constants; $v_0$ 
and $v_1$ are scalar control inputs.
The problem \eqref{deq1.1} describes the unidirectional propagation of nonlinear 
dispersive long waves with dissipative effect. In case, $\nu=0$ and $\mu>0,$ the equation \eqref{deq1.1} 
is known as Benjamin-Bona-Mahony (BBM) equation.
When $\mu=0$ and $\nu>0$  in \eqref{deq1.1}, then it is called Burgers' equation. 
For mathematical modeling and physical applications of \eqref{deq1.1}, see \cite{Peregrine}, \cite{Benjamin}, 
\cite{Bona} and references, therein.

Based on distributed and Dirichlet boundary control in feedback form through Riccati operator,
local stabilization results for the Burgers' equation with sufficiently small initial data 
are established in \cite{bk}, \cite{bk1}. Moreover, for local stabilization results using Neumann boundary 
control, we refer to  \cite{bgs}, 
\cite{ik} and \cite{iy}. It is to be noted that for viscous Burgers' equation, global existence 
and uniqueness results with Dirichlet  and Neumann boundary conditions  are derived for any 
initial data in $L^2$ in \cite{lmt}.
Subsequently, based on nonlinear Neumann and Dirichlet boundary control laws, global stabilization 
results for the Burgers' equation are proved using a suitable application Lyapunov type functional  
in Krstic \cite{krstic1}, Balogh and Krstic \cite{balogh}. Later on, adaptive (when $\nu$ is unknown) 
and nonadaptive (when $\nu$ is known) stabilization results for generalized Burgers' equations are 
established in \cite{liu}, \cite{Smaoui} and \cite{smaoui1} with different types of boundary conditions. 
For existence of solution to the problem \eqref{deq1.1}-\eqref{deq1.4}, when $\mu=0$, we refer 
to \cite{balogh} and \cite{liu}.

For stabilization of the BBM-B equation, the authors in \cite{hasan} have shown global stabilization results corresponding to $\mu=1$ with zero Dirichlet boundary condition at one end and Neumann boundary control on the other end. Using a reduced order model, distributed feedback control for the BBM-B equation is discussed in \cite{piao1}.  Also, quadratic $B$-spline finite element method followed by linear quadratic regulator theory to design feedback control, is used to stabilize in \cite{piao2} without any convergence analysis.
In \cite{skakp}, we have shown that, under the uniqueness assumption of the steady state solution, the steady 
state solution of the problem \eqref{deq1.1} with zero Dirichlet boundary condition is exponentially 
stable. 

In this paper, we discuss global stabilization results using nonlinear Neumann feedback control law. Our second objective is to apply $C^0$-finite element method to the stabilization problem \eqref{deq1.1}-\eqref{deq1.4} using nonlinear Neumann boundary control laws and discuss convergence analysis. Since to the best of our knowledge, there is hardly any discussion in the literature on the rate of convergence, hence, in this paper, an effort has been made to prove optimal order of convergence of the state variable along with superconvergence result for the feedback control laws.
The main contributions of this article are summarized as:
\begin{itemize}
\item Global stabilization for  problem \eqref{deq1.1}-\eqref{deq1.4}, that is, convergence of the unsteady solution to the problem \eqref{deq1.1} to its constant steady state solution under nonlinear Neumann boundary control laws \eqref{deq1.2}-\eqref{deq1.3} is proved.
\item Based on the $C^0$- conforming finite element method, global stabilization results for the semidiscrete solution are discussed and optimal error estimates are established in $L^\infty(L^2)$, $L^\infty(H^1)$, and $L^\infty(L^\infty)$ norms for the state variable. Moreover, superconvergence results are derived for the nonlinear Neumann feedback control laws. 
\item  
Finally, some numerical experiments are conducted to confirm our theoretical results.
\end{itemize}
For related issues of finite element analysis 
of the viscous Burgers' equation using nonlinear Neumann boundary feedback control law, we refer to our 
recent article \cite{skakp1}. Compared to \cite{skakp1}, special care has been taken to establish global stabilization 
results in $L^\infty(H^i) (i=0, 1, 2)$ norms as $\mu\to 0.$ It is further observed that the decay rate for the BBM-B type equation is less than 
the decay rate for the viscous Burgers' equation and as the  dispersion coefficient $\mu$ approaches zero,
the decay rate also converges to the decay rate for the Burgers' equation.
Finite element error analysis holds for fixed $\mu$.

For the rest of this article, we denote by $H^m(I=[0,1])$  the standard Sobolev spaces with norm $\norm{\cdot}_m$ and for $m=0$, $\norm{\cdot}$ denotes the corresponding $L^2$ norm.
The space $L^p((0,T);X)$  $1\leq p\leq\infty, $ consists of all strongly measurable functions $v:(0,T) \rightarrow X $ with norm
 $$\norm{v}_{L^p((0,T);X)}:=\left(\int_{0}^{T}\norm{v(t)}_X^p dt\right)^\frac{1}{p}<\infty \quad \text{for} \quad 1\leq p<\infty,$$ and 
$$\norm{v}_{L^\infty((0,T);X)}:=\operatorname*{ess\,sup}\limits_{0\leq t\leq T}\norm{v(t)}_X<\infty.$$ When there is no confusion, $L^p((0,T);X)$ is simply denoted by $L^p(X)$.
The equilibrium or steady state solution $u^\infty$ of \eqref{deq1.1}-\eqref{deq1.3} satisfies
\begin{align}
	&-\nu u^\infty_{xx}+u^\infty_x+u^\infty u^\infty_{x}=0 \qquad\text{in} \quad (0,1) \label{deq1.5},\\
	&u^\infty_x(0)=u^\infty_x(1)=0 \label{deq1.6}.
\end{align}
Note that any constant $w_d$ is a solution of the steady state problem \eqref{deq1.5}-\eqref{deq1.6}. Without loss of generality, we assume that $w_d\geq 0$.\\
Set $w=u-w_d,$ which satisfies
\begin{align}
&w_t-\mu w_{xxt}-\nu w_{xx}+(1+w_d)w_x+ww_x=0 \qquad\text{in} \quad (0,1)\times(0,\infty)\label{deq1.7} ,\\
&w_x(0,t)=v_0(t)  \label{deq1.8}\quad t\in(0,\infty),\\
&w_x(1,t)=v_1(t)  \label{deq1.9}\quad t\in(0,\infty),\\
&w(x,0)=w_0(x), \label{deq1.10}\quad x\in(0,1),
\end{align}
where, $w$ is the state variable and $v_0$ and $v_1$ are feedback control variables.
Since for the problem with zero Neumann boundary condition, the steady state constant solution 
$w_d$ is not asymptotically  stable, we plan to achieve stabilization result through boundary feedback law.
The present analysis can be easily extended to the problem with one side control law  say for  example:
when $w(0,t)=0$, $w_x(1,t)=v_1(t)$, see \cite{hasan}.
The weak formulation of the problem \eqref{deq1.7}-\eqref{deq1.10} is to seek $w(t)\in H^1(0,1),$ $w_t \in 
L^2(L^2)$ and $\mu w_t\in L^2(H^1)$
such that  for almost all $t>0$
\begin{align}
(w_t,\chi)+&\mu(w_{xt},\chi_x)+\nu(w_x,\chi_x)+(1+w_d)(w_x,\chi)+(ww_x,\chi)-\mu\Big(v_{1t}(t)\chi(1)\notag\\
&-v_{0t}(t)\chi(0)\Big)-\nu\Big(v_1(t)\chi(1)-v_0(t)\chi(0)\Big)=0  \qquad \forall~ \chi\in H^1
\label{deq1.11}
\end{align}
with $w(x,0)=w_0(x).$
For motivation to choose the control laws $v_{0t}$ and $v_{1t}$ using construction of Lyapunov functional, 
see \cite{krstic1}.
Based on the nonlinear Neumann control law propose in our earlier article in Burgers' equation, see 
\cite{skakp1}, which is a modification of control law in \cite{krstic1}, we now choose  the feedback 
control law as  
\begin{equation}\label{deq1.13}
	w_x(0,t)=v_0(t)=:\frac{1}{\nu}\Big((c_0+1+w_d)w(0,t)+\frac{2}{9c_0}w^3(0,t)\Big)\equiv: K_0(w(0,t)),
\end{equation}
and 
\begin{equation}\label{deq1.14}
	w_x(1,t)=v_1(t)=:-\frac{1}{\nu}\Big((c_1+1+w_d)w(1,t)+\frac{2}{9c_1}w^3(1,t)\Big)\equiv:K_1(w(1,t)),
\end{equation}
where $K_0$ and $K_1$ represent feedback control laws, and $c_0$ and $c_1$ are positive constants.\\
Using \eqref{deq1.13}-\eqref{deq1.14}, we obtain a typical nonlinear problem \eqref{deq1.7}-\eqref{deq1.10} with boundary conditions \eqref{deq1.13}-\eqref{deq1.14}.
Its weak formulation \eqref{deq1.11} becomes
\begin{align}
(w_t,\chi)+&\mu(w_{xt},\chi_x)+\nu(w_x,\chi_x)+(1+w_d)(w_x,\chi)+(ww_x,\chi)\notag\\
&+\frac{\mu}{\nu}\Big(\big((c_0+1+w_d)w_t(0,t)+\frac{2}{3c_0}w^2(0,t)w_t(0,t)\big)\chi(0)+\big((c_1+1+w_d)w_t(1,t)\notag\\
&+\frac{2}{3c_1}w^2(1,t)w_t(1,t)\big)\chi(1)\Big)+\Big(\big((c_0+1+w_d)w(0,t)\notag\\
&+\frac{2}{9c_0}w^3(0,t))\chi(0)+\big((c_1+1+w_d)w(1,t)+\frac{2}{9c_1}w^3(1,t)\big)\chi(1)\Big)=0  \qquad \forall~ \chi\in H^1
\label{deqx1.11}.
\end{align}
Throughout the paper, we use the following norm which is equivalent to the usual $H^1$-norm:
\begin{equation}\label{eqxn1.11}
\tnorm{z(t)}=\sqrt{z^2(0,t)+z^2(1,t)+\norm{z_x(t)}^2},
\end{equation}
and $C$ is used as a generic positive constant.

We now recall some results to be use in our subsequent sections.
\begin{lemma}\label{pw}
{\bf{Poincar\'e-Wirtinger's inequality}} For any $z(t)\in H^1(0,1)$, the following inequality holds:
$$\norm{z(t)}^2\leq 2z^2(i,t)+\norm{z_x(t)}^2, \quad \text{for} \hspace{0.1cm} i=0 \hspace{0.1cm}\text {or} 
\hspace{0.1cm} 1.$$
\end{lemma}
Using Agmon's and Poincar\'e inequality, the following inequality holds
\begin{equation}\label{eqx1.12}
\norm{z(t)}_{L^\infty}\leq \sqrt{2}\tnorm{z(t)},
\end{equation}
where $\tnorm{\cdot}$ is given in \eqref{eqxn1.11}.\\
Bellow, we assume the following well posedness theorem for the problem \eqref{deq1.7}-\eqref{deq1.10}
\begin{theorem}\label{thmx1}
Let  $w_0(x)\in H^2(0,1).$ Then, there exists a unique weak solution 
$w(t)\in H^1(0,1)$, $w_t\in L^2(L^2)$ and $\mu w_t\in L^2(H^1)$ of \eqref{deq1.7}-\eqref{deq1.8} 
satisfying the weak formulation \eqref{deqx1.11}.

In addition, the following regularity result holds
\begin{align}
\norm{w(t)}^2_2+ \norm{w_t(t)}^2+\mu \norm{w_{xt}(t)}^2+\int_{0}^{t}
\Big(\norm{w_t(s)}^2_1+\mu\norm{w_{xxt}(s)}^2\Big)\;ds\leq C.
\end{align}
\end{theorem}
Subsequently for our error estimates in $L^\infty(L^\infty)$ norm, we further assumed 
that $w(t)\in W^{2,\infty}$ with its norm denoted by $\norm{\cdot}_{2,\infty}$.

The rest of the article is organized as follows. Section $2$ deals with global stabilization results and the existence and uniqueness of strong solution. Section $3$ is devoted to optimal error estimates for the semidiscrete solution with superconvergence results for feedback controllers. Finally in section $4$, some numerical examples
are considered to confirm our theoretical results.
\section{Stabilization and continuous dependence result}
In this subsection, we discuss {\it {a priori}} bounds for the problem \eqref{deqx1.11} and derive stabilization results. In addition, these estimates are needed to prove optimal error estimates for the state variable and  feedback controllers.
All estimates throughout the paper  are valid for the same $\alpha$ with
\begin{align}\label{denq1.1}
0\leq \alpha\leq \frac{1}{2}\min\Bigg\{\frac{\nu}{(\mu+1)}, \frac{\nu}{(2\mu+\nu)}, 
\frac{\nu(1+c_i+w_d)}{(\nu+(1+c_i+w_d)\mu)} (i=0,1)\Bigg\}.
\end{align}
\begin{lemma}\label{dlm1.1}
	Let $w_0\in H^1(0,1)$. Then, there holds
	\begin{align*}
	\norm{w(t)}^2+&\mu\norm{w_{x}(t)}^2+\frac{\mu}{\nu}E_1(w)(t)+\beta e^{-2\alpha t}\int_{0}^{t} e^{2\alpha s}\Big(E_1(w)(s)+\norm{w_x(s)}^2\Big) ds\\
	&\leq e^{-2\alpha t}\big(\norm{w_0}^2+\mu\norm{w_{0x}}^2+\frac{\mu}{\nu}E_1(w)(0)\big),
	\end{align*}
	where $\alpha$ is given in \eqref{denq1.1},  
	\begin{equation}\label{deqn1.1}
	\beta=\min\Bigg\{2\big(\nu-\alpha(\mu+1)\big), \big(1-2\alpha\frac{\mu}{\nu}\big),
	\Big((1+c_i+w_d)-2\alpha\big((1+c_i+w_d)\frac{\mu}{\nu}+1\big)\Big),i=0,1 \Bigg\},
	\end{equation}
	and 
	\begin{equation}\label{deq1.17}
	E_1(w)(t)=\sum_{i=0}^{1}\Big((c_i+1+w_d)+\frac{1}{3c_i}w^2(i,t)\Big)w^2(i,t).
	\end{equation}
\end{lemma}
\begin{proof}
	Set $\chi=w$ in the weak formulation \eqref{deqx1.11} to obtain
	\begin{align}
	\frac{d}{dt}\Big( \norm{w(t)}^2+&\mu\norm{w_x(t)}^2 +\frac{\mu}{\nu}E_1(w)(t)\Big)+2\nu\norm{w_x(t)}^2+E_1(w)(t)+\Big(c_0w^2(0,t)+\frac{1}{9c_0}w^4(0,t)\notag\\
	&+\big(c_1+2(1+w_d)\big)w^2(1,t)+\frac{1}{9c_1}w^4(1,t)\Big)=\frac{2}{3}\Big(w^3(0,t)-w^3(1,t)\Big)\label{deq1.15},
	\end{align}
	where $E_1(w)(t)$ is given in \eqref{deq1.17}.
	A use of Young's inequality for the right hand side term shows
	\begin{equation}\label{deq1.16}
	\frac{2}{3}w^3(i,t)\leq c_i w^2(i,t)+\frac{1}{9c_i}w^4(i,t), \quad i=0,1.
	\end{equation}
	Therefore, using \eqref{deq1.16} and \eqref{deq1.17}, we obtain from \eqref{deq1.15}
	\begin{align}
	\frac{d}{dt}\Big(\norm{w(t)}^2+\mu\norm{w_x(t)}^2&+\frac{\mu}{\nu}E_1(w)(t)\Big)\notag\\
	&\qquad+E_1(w)(t)+2(1+w_d)w^2(1,t)+2\nu\norm{w_x}^2\leq 0\label{deq1.18}.
	\end{align}
	Multiply \eqref{deq1.18} by $e^{2\alpha t}$ to arrive at
	\begin{align}
	\frac{d}{dt}\Big(e^{2\alpha t}\big(\norm{w(t)}^2+&\mu\norm{w_x(t)}^2+\frac{\mu}{\nu}E_1(w)(t)\big)\Big)-2\alpha e^{2\alpha t}\big(\norm{w(t)}^2+\mu\norm{w_x(t)}^2+\frac{\mu}{\nu}E_1(w)(t)\big)\notag\\
	&+e^{2\alpha t}\Big(E_1(w)(t)+2\nu\norm{w_x(t)}^2+2(1+w_d)w^2(1,t)\Big)\leq 0\label{deq1.19}.
	\end{align}
	A use of Poincar\'{e}-Wirtinger's inequality yields
	\begin{equation}\label{deq1.20}
	\norm{w(t)}^2\leq w^2(0,t)+w^2(1,t)+\norm{w_x(t)}^2.
	\end{equation}
	Substitute \eqref{deq1.20} in \eqref{deq1.19} and expanding $E_1(w)(t)$ to find that
	\begin{align}
	\frac{d}{dt}\Big(e^{2\alpha t}&\big(\norm{w(t)}^2+\mu\norm{w_x(t)}^2+\frac{\mu}{\nu}E_1(w)(t)\big)\Big)\notag\\
	&+e^{2\alpha t} \Bigg(\sum_{i=0}^{1}\Big((1+c_i+w_d)-2\alpha\big((1+c_i+w_d)\frac{\mu}{\nu}+1\big)\Big)w^2(i,t)\notag\\
	&\quad+\Big(1-2\alpha\frac{\mu}{\nu}\Big)\sum_{i=0}^{1}\frac{1}{3c_0}w^4(i,t)\Bigg)+2\big(\nu-\alpha(\mu+1)\big)e^{2\alpha t}\norm{w_x(t)}^2
\leq 0 \label{deq1.21}.
	\end{align}
	Now choose $\alpha$ as in \eqref{denq1.1}, so that all the coefficients on the left hand side are positive.
	Then integrating the above inequality from $0$ to $t$ and multiplying the resulting inequality by $e^{-2\alpha t},$ we obtain
	\begin{align*}
	\norm{w(t)}^2+&\mu\norm{w_x(t)}^2+\frac{\mu}{\nu}E_1(w)(t)
	+\beta e^{-2\alpha t}\int_{0}^{t} e^{2\alpha s}\Big(E_1(w)(s)+\norm{w_x(s)}^2\Big)\; ds\\
	&\hspace{5cm}\leq e^{-2\alpha t}\big(\norm{w_0}^2+\mu\norm{w_{0x}}^2+\frac{\mu}{\nu}E_1(w)(0)\big).
	\end{align*}
	This completes the proof.
\end{proof}
\begin{remark}\label{drm1.1}
	Since 
	\begin{align*}
	E_1(w)(t)+\norm{w_x(t)}^2\geq \tnorm{w(t)}^2,
	\end{align*}
	we obtain from Lemma \ref{dlm1.1}
	\begin{align*}
	\beta e^{-2\alpha t}\int_{0}^{t}e^{2\alpha s}\tnorm{w(s)}^2\; ds\leq e^{-2\alpha t}\Big(\norm{w_0}^2+\mu\norm{w_{0x}}^2+\frac{\mu}{\nu}E_1(w)(0)\Big).
	\end{align*}
	When $\alpha=0,$ Lemma \ref{dlm1.1} holds for all $t>0$, that is,
	\begin{equation*}
	\int_{0}^{t}\tnorm{w(s)}^2 \;ds \leq \Big(\norm{w_0}^2+\mu\norm{w_{0x}}^2+E_1(w)(0)\Big)\leq C.
	\end{equation*}
\end{remark}
\begin{lemma}\label{dlm1.2}
	Let $w_0\in H^2(0,1)$. Then, there holds
	\begin{align*}
	\Big(\norm{w_x(t)}^2+&\mu\norm{w_{xx}(t)}^2+\frac{1}{\nu}E_2(w)(t)\Big)+\beta e^{-2\alpha t}\int_{0}^{t} e^{2\alpha s}\norm{w_{xx}(s)}^2 ds\\&\quad\leq Ce^{-2\alpha t}\big(\norm{w_{0x}}^2+\mu\norm{w_{0xx}}^2+\frac{1}{\nu}E_1(w)(0)\big)e^C,
	\end{align*}
	where $E_2(w)(t)=\Big((c_0+1+w_d)+\frac{1}{9c_0}w^2(0,t)\Big)w^2(0,t)+\Big((c_1+1+w_d)
	+\frac{1}{9c_1}w^2(1,t)\Big)w^2(1,t).$
\end{lemma}
\begin{proof}
	Forming the $L^2$- inner product between \eqref{deq1.7} and $-w_{xx},$ we obtain
	\begin{align}
	\frac{d}{dt}\Big(\norm{w_x(t)}^2+\mu\norm{w_{xx}(t)}^2\Big)&+2\nu \norm{w_{xx}(t)}^2-2\big(v_1(t)w_t(1,t)-v_0(t)w_t(0,t)\big)\notag\\
	&\qquad =2(ww_x,w_{xx})+2(1+w_d)(w_x,w_{xx})\label{deq1.23}.
	\end{align}
	After substituting \eqref{deq1.13}-\eqref{deq1.14} in \eqref{deq1.23}, the contributions of the boundary terms in \eqref{deq1.23} are 
	\begin{align}
	-2\big(w_t(1,t)w_x(1,t)-&w_t(0,t)w_x(0,t)\big)=\frac{1}{\nu}\frac{d}{dt}\Big((c_0+1+w_d)w^2(0,t)+(c_1+1+w_d)w^2(1,t)\notag\\
	&\qquad+\frac{1}{9c_0}w^4(0,t)+\frac{1}{9c_1}w^4(1,t)\Big)=\frac{1}{\nu}\frac{d}{dt}(E_2(w)(t))\label{deq1.24}.
	\end{align}
	The terms on the right hand side of \eqref{deq1.23} are now bounded by
	\begin{equation*}
	2(1+w_d)(w_x,w_{xx})\leq \frac{\nu}{2}\norm{w_{xx}(t)}^2+\frac{2}{\nu}(1+w_d)^2\norm{w_{x}(t)}^2,
	\end{equation*}
	and
	\begin{equation*}
	2(ww_x,w_{xx})\leq \norm{w(t)}_{L^\infty}\norm{w_x(t)}\norm{w_{xx}(t)}\leq \frac{\nu}{2}\norm{w_{xx}(t)}^2+C\tnorm{w(t)}^2\norm{w_x(t)}^2.
	\end{equation*}
Using \eqref{deq1.23} we arrive  at
	\begin{align}
	\frac{d}{dt}\Big(\norm{w_x(t)}^2&+\mu\norm{w_{xx}(t)}^2+\frac{1}{\nu}E_2(w)(t)\Big)+\nu\norm{w_{xx}(t)}^2\notag\\
	&\leq C\norm{w_x(t)}^2
	+C\tnorm{w(t)}^2\norm{w_x(t)}^2\label{deq1.25}.
	\end{align}
	Multiplying the above inequality by $e^{2\alpha t},$ and using
	\begin{equation*}
	\norm{w_x(t)}^2\leq w_x^2(0,t)+w_x^2(1,t)+\norm{w_{xx}(t)}^2\leq C\Big(1+w^2(0,t)+w^2(1,t)\Big)E_2(w)(t)+\norm{w_{xx}(t)}^2,
	\end{equation*}
	and $E_2(w)(t)\leq E_1(w)(t)$
	we obtain
	\begin{align*}
	\frac{d}{dt}\Big(e^{2\alpha t}\big(\norm{w_x(t)}^2&+\mu\norm{w_{xx}(t)}^2+\frac{1}{\nu}E_2(w)(t)\big)\Big)+\Big(\nu-2\alpha (\mu+1)\Big) e^{2\alpha t}\norm{w_{xx}(t)}^2\\
	&\leq C e^{2\alpha t}\big(\norm{w_x(t)}^2+\frac{1}{\nu}E_2(w)(t)\big)\\
	&\qquad+Ce^{2\alpha t}\tnorm{w(t)}^2\Big(\norm{w_x(t)}^2+\mu\norm{w_{xx}(t)}^2+\frac{1}{\nu}E_2(w)(t)\Big)\\
	&\quad\leq Ce^{2\alpha t}(\norm{w_x(t)}^2+E_1(w)(t))\\
	&\qquad+C\tnorm{w(t)}^2e^{2\alpha t}\Big(\norm{w_x(t)}^2+\mu\norm{w_{xx}(t)}^2+\frac{1}{\nu}E_2(w)(t)\Big).
	\end{align*}
	A use of Gronwall's inequality now yields
	\begin{align}
	e^{2\alpha t}\Big(\norm{w_x(t)}^2+&\mu\norm{w_{xx}(t)}^2+\frac{1}{\nu}E_2(w)(t)\Big)+\Big(\nu-2\alpha(\mu+1)\Big)\int_{0}^{t} e^{2\alpha s}\norm{w_{xx}(s)}^2 ds\notag\\
	&\leq \Big(\norm{w_{0x}}^2+\mu\norm{w_{0xx}}^2+\frac{1}{\nu}E_2(w)(0)+C\int_{0}^{t}e^{2\alpha s}\big(\norm{w_x(t)}^2+E_1(w)(s)\big) ds\Big)\notag\\
	&\qquad\exp\Big(C\int_{0}^{t}\tnorm{w(s)}^2 ds\Big)\label{deq1.29}.
	\end{align}
	From Remark \ref{drm1.1} and Lemma \ref{dlm1.1}, we bound the right hand side term of \eqref{deq1.29}.
	Therefore, after multiplying \eqref{deq1.29} by $e^{-2\alpha t},$ we obtain
	\begin{align*}
	\Big(\norm{w_x(t)}^2+&\mu\norm{w_{xx}(t)}^2+\frac{1}{\nu}E_2(w)(t)\Big)+\beta e^{-2\alpha t}\int_{0}^{t} e^{2\alpha s}\norm{w_{xx}(s)}^2 ds\\&\quad\leq Ce^{-2\alpha t}\big(\norm{w_{0x}}^2+\mu\norm{w_{0xx}}^2+\frac{1}{\nu}E_1(w)(0)\big)\\
	&\qquad\exp\Big(C(\norm{w_0}^2+\mu\norm{w_{0x}}^2+\frac{\mu}{\nu}E_1(w)(0))\Big).
	\end{align*}
	Since the terms in the bracket in the exponential form are bounded, this
	completes the rest of the proof.
\end{proof}
\begin{lemma}\label{dlm1.3}
	Let $w_0\in H^2(0,1)$. Then, there holds
	\begin{align*}
	\nu\big(\norm{w_x(t)}^2&+E_2(w)(t)\big)+e^{-2\alpha t}\int_{0}^{t}e^{2\alpha s}\Big(\norm{w_t(t)}^2+\mu\norm{w_{xt}(t)}^2+\frac{\mu}{\nu}E_3(w)(s)\Big) ds\\
	&\leq C e^Ce^{-2\alpha t}\big(\norm{w_{0x}}^2+\mu\norm{w_{0xx}}^2+\frac{1}{\nu}E_1(w)(0)\big),
	\end{align*}
	where
	\begin{align}\label{deqn1.3}
	E_3(w)(t)&=\Big((1+c_0+w_d)+\frac{2}{3c_0}w^2(0,t)\Big)w_t^2(0,t)+\Big((1+c_1+w_d)\notag\\
	&\qquad+\frac{2}{3c_1}w^2(1,t)\Big)w_t^2(1,t).
	\end{align}
\end{lemma}
\begin{proof}
	Set $\chi=w_t$ in the weak formulation \eqref{deqx1.11} to obtain
	\begin{align}
	\norm{w_t(t)}^2+&\mu\norm{w_{xt}(t)}^2+\frac{1}{2}\frac{d}{dt}\Big(\nu\norm{w_x(t)}^2+E_2(w)(t)\Big)+\frac{\mu}{\nu}E_3(w)(t)\notag\\
	&\hspace{5cm}=(1+w_d)(w_x,-w_t)+(ww_x,-w_t),\label{deq1.32}
	\end{align}
	where $E_3(w)(t)$ is given in \eqref{deqn1.3}.
	Note that
	$$
	(1+w_d)(w_x,-w_t)\leq \frac{1}{4}\norm{w_t(t)}^2+(1+w_d)^2\norm{w_x(t)}^2,
	$$
	and
	$$(ww_x,-w_t)\leq C\norm{w(t)}_{L^\infty}\norm{w_x(t)}\norm{w_t(t)}\leq \frac{1}{4}\norm{w_t(t)}^2+C\norm{w_x(t)}^2\tnorm{w(t)}^2.$$
	Therefore, from \eqref{deq1.32}, we arrive at
	\begin{align*}
	\frac{d}{dt}\Big(\nu\norm{w_x(t)}^2+E_2(w)(t)\Big)&+\norm{w_t(t)}^2+2\mu\norm{w_{xt}(t)}^2+\frac{2\mu}{\nu}E_3(w)(t)\leq C\norm{w_x(t)}^2\tnorm{w(t)}^2+C\norm{w_x(t)}^2.
	\end{align*}
	Multiply the above inequality by $e^{2\alpha t}$. Now, a use of the Gronwall's inequality and Lemma \ref{dlm1.1} completes the rest of the proof.
\end{proof}
\begin{lemma}\label{dlm1.4}
	Let $w_0\in H^2(0,1)$. Then, there holds
	\begin{align*}
	\Big(\norm{w_t(t)}^2+\mu\norm{w_{xt}(t)}^2&+\frac{\mu}{\nu}E_3(w)(t)\Big)+e^{-2\alpha t}\int_{0}^{t}e^{2\alpha s}\Big(\nu\norm{w_{xt}(s)}^2+E_3(w)(s)\Big) ds\\
	&\leq
	Ce^Ce^{-2\alpha t}\Big(\big(\norm{w_0}^2+\mu\norm{w_{0x}}^2+\frac{\mu}{\nu}E_1(w)(0)\big)\Big),
	\end{align*}
	where $E_3(w)(t)$ is as in \eqref{deqn1.3}.
\end{lemma}
\begin{proof}
	Differentiating \eqref{deq1.7} with respect to $t$ and then taking the inner product with $\chi=w_t$, we obtain
	\begin{align}
	\frac{1}{2}\frac{d}{dt}\Big(\norm{w_t(t)}^2+&\mu\norm{w_{xt}(t)}^2+\frac{\mu}{\nu}E_3(w)(t)\Big)+\nu\norm{w_{xt}(t)}^2+(1+w_d)(w_{xt},w_t)+(w_tw_x+ww_{xt},w_t)\notag\\
	&+2\frac{\mu}{\nu}\Big(\frac{1}{3c_0}w(0,t)w_t^3(0,t)+\frac{1}{3c_1}w(1,t)w_t^3(1,t)\Big)+E_3(w)(t)=0\label{deq6.1}.
	\end{align}
	The other terms in \eqref{deq6.1} are bounded by
	\begin{align*}
	(1+w_d)(w_{xt},w_t)&\leq \frac{\nu}{4}\norm{w_{xt}(t)}^2+\frac{2(1+w_d^2)}{\nu}\norm{w_t(t)}^2,\\
	(w_tw_x+ww_{xt},w_t)&\leq
	\norm{w_t(t)}_{L^\infty}\norm{w_x(t)}\norm{w_t(t)}+\norm{w(t)}_{L^\infty}\norm{w_{xt}(t)}\norm{w_t(t)}\\
	&\leq (\lvert{w_t(0,t)}\rvert+\norm{w_{xt}(t)})\norm{w_x(t)}\norm{w_t(t)}+\sqrt{2}\tnorm{w(t)}\norm{w_{xt}(t)}\norm{w_t(t)}\\
	&\leq \frac{1+c_0+w_d}{2}w_t^2(0,t)+\frac{\nu}{4}\norm{w_{xt}(t)}^2+C\tnorm{w(t)}^2\norm{w_t(t)}^2,
	\end{align*}
	and
	\begin{equation*}
	\frac{2\mu}{\nu}\frac{1}{3c_0}w(0,t)w_t^3(0,t)\leq \frac{1}{3c_0}w^2(0,t)w_t^2(0,t)+\frac{\mu^2}{\nu^2}\frac{1}{3c_0}w_t^4(0,t).
	\end{equation*}
	Therefore, from \eqref{deq6.1}, we arrive at
	\begin{align*}
	\frac{d}{dt}\Big(\norm{w_t}^2+&\mu\norm{w_{xt}}^2+\frac{\mu}{\nu}E_3(w)(t)\Big)+\nu\norm{w_{xt}(t)}^2+\Big((1+c_0+w_d)w_t^2(0,t)\\
	&+2(1+c_1+w_d)w_t^2(1,t)+\frac{2}{3c_0}w^2(0,t)w_t^2(0,t)+\frac{2}{3c_1}w^2(1,t)w_t^2(1,t)\Big)\\
	&\leq
	2\frac{\mu^2}{\nu^2}\Big(\frac{1}{3c_0}w_t^4(0,t)+\frac{1}{3c_1}w_t^4(1,t)\Big) +C\tnorm{w(t)}^2\norm{w_t(t)}^2+C\norm{w_t(t)}^2.
	\end{align*}
	Now multiply the above inequality by $e^{2\alpha t}$ to obtain
	\begin{align*}
	\frac{d}{dt}\Big(e^{2\alpha t}\big(\norm{w_t(t)}^2&+\mu\norm{w_{xt}(t)}^2+\frac{\mu}{\nu}E_3(w)(t)\big)\Big)+\nu e^{2\alpha t}\norm{w_{xt}(t)}^2+e^{2\alpha t}E_3(w)(t)\\
	&\leq C e^{2\alpha t}\tnorm{w(t)}^2\norm{w_t(t)}^2+C(\alpha) e^{2\alpha t}\Big(\norm{w_t(t)}^2+\mu\norm{w_{xt}(t)}^2+\frac{\mu}{\nu}E_3(w)(t)\Big)\\
	&\quad+C\frac{\mu}{\nu}e^{2\alpha t}E_3(w)(t)\Big(w_t^2(0,t)+w_t^2(1,t)\Big)\\
	&\leq Ce^{2\alpha t}\Big(\norm{w_t(t)}^2+\mu\norm{w_{xt}(t)}^2+\frac{\mu}{\nu}E_3(w)(t)\Big)
	\big(\tnorm{w(t)}^2+\frac{\mu}{\nu}(w_t^2(0,t)+w_t^2(1,t)\big)\\
	&\qquad+Ce^{2\alpha t}\Big(\norm{w_t(t)}^2+\mu\norm{w_{xt}(t)}^2+\frac{\mu}{\nu}E_3(w)(t)\Big).
	\end{align*}
	By the Gronwall's inequality, it follows from above with a use of Lemmas \ref{dlm1.1} and \ref{dlm1.3} that
	\begin{align*}
	e^{2\alpha t}\Big(\norm{w_t(t)}^2&+\mu\norm{w_{xt}(t)}^2+\frac{\mu}{\nu}E_3(w)(t)\Big)+\int_{0}^{t}e^{2\alpha s}\big(\nu\norm{w_{xt}(s)}^2+E_3(w)(s)\big) ds\\
	&\leq \Big(\norm{w_t(0)}^2+\mu\norm{w_{xt}(0)}^2+\frac{\mu}{\nu}E_3(w)(0)+C\int_{0}^{t}e^{2\alpha s} \Big(\norm{w_t(t)}^2+\mu\norm{w_{xt}(t)}^2\\
	&\qquad+\frac{\mu}{\nu}E_3(w)(s)\Big)ds\Big)\exp\Big(C\int_{0}^{t}\big(\tnorm{w(t)}^2+\frac{\mu}{\nu}(w_t^2(0,s)+w_t^2(1,s)) \big)ds\Big)\\
	&\leq C\Big(\norm{w_t(0)}^2+\mu\norm{w_{xt}(0)}^2+\frac{\mu}{\nu}E_3(w)(0)+\big(\norm{w_{0x}}^2+\mu\norm{w_{0xx}}^2+\frac{1}{\nu}E_1(w)(0)\big)\\
	&
	\qquad\exp\Big(C\big(\norm{w_0}^2+\mu\norm{w_{0x}}^2+\frac{\mu}{\nu}E_1(w)(0)\big)\Big).
	\end{align*}
	Also after putting $\chi=w_t$ in the weak formulation \eqref{deqx1.11}, we arrive at
	$$\norm{w_t(t)}^2+\mu\norm{w_{xt}(t)}^2+\frac{\mu}{\nu}E_3(w)(t)\leq 3\nu \norm{w_{xx}(t)}^2+C\norm{w_x(t)}^2+C\tnorm{w(t)}^2\norm{w_x(t)}^2.$$
	Therefore, we can find the value of $\norm{w_t(t)}^2+\mu\norm{w_{xt}(t)}^2+\frac{\mu}{\nu}E_3(w)(t)$ at $t=0$ as
	\begin{align*} \norm{w_t(0)}^2&+\mu\norm{w_{xt}(0)}^2+\frac{\mu}{\nu}E_3(w)(0)\\
	&\leq C\Big(\norm{w_{0x}}^2+\mu\norm{w_{0xx}}^2+\frac{1}{\nu}E_1(w)(0)\Big)\exp\Big(C\big(\norm{w_0}^2+\mu\norm{w_{0x}}^2+\frac{\mu}{\nu}E_1(w)(0)\big)\Big).
	\end{align*}
	Hence, we arrive at
	\begin{align*}
	e^{2\alpha t}\Big(\norm{w_t(t)}^2&+\mu\norm{w_{xt}(t)}^2+\frac{\mu}{\nu}E_3(w)(t)\Big)+\int_{0}^{t}e^{2\alpha s}\big(\nu\norm{w_{xt}(s)}^2+E_3(w)(s)\big) ds\\
	&\leq
	\Big(C\big(\norm{w_{0x}}^2+\mu\norm{w_{0xx}}^2+\frac{\mu}{\nu}E_1(w)(0)\big)\Big)\\
	&\qquad\exp\Big(C\big(\norm{w_{0x}}^2+\mu\norm{w_{0x}}^2+\frac{\mu}{\nu}E_1(w)(0)\big)\Big).
	\end{align*}
	Multiply the above inequality by $e^{-2\alpha t}$ to complete the proof.
\end{proof}
\begin{lemma}\label{dlmx1.5}
	Let $w_0\in H^2(0,1)$. Then, there holds
	\begin{align*}
	\nu\norm{w_{xx}(t)}^2&+e^{-2\alpha t}\int_{0}^{t}e^{2\alpha s}\Big(\norm{w_{xt}(t)}^2+\mu\norm{w_{xxt}(t)}^2+\frac{2}{\nu}E_3(w)(s)\Big)\;ds\\
	&\leq Ce^Ce^{-2\alpha t}\Big(C(1+\mu)\big(\norm{w_0}^2_2+E_1(w)(0)\big)\Big).
	\end{align*}
\end{lemma}
\begin{proof}
	Form the $L^2$-inner product between \eqref{deq1.7} and $-w_{xxt}$ to obtain
	\begin{align}
	\norm{w_{xt}(t)}^2&+2\mu\norm{w_{xxt}(t)}^2+\frac{2}{\nu}E_3(w)(t)+\nu\frac{d}{dt}\norm{w_{xx}(t)}^2\notag\\
	&\leq C\Big(1+\tnorm{w(t)}^2\Big)\norm{w_{xx}(t)}^2+CE_1(w)(t)\norm{w_x(t)}^2+C\Big(E_1(w)(t)+E_3(w)(t)\Big)\label{deqn2.1},
	\end{align}
	where we use the bound of $w^2(i,t)$ and $w^4(i,t)$ for $i=0,\hspace{0.1cm}1$ from Lemma \ref{dlm1.2}.\\
	Multiply \eqref{deqn2.1} by $e^{2\alpha t}$ to obtain
	\begin{align*}
	\frac{d}{dt}\Big(e^{2\alpha t}&\big(\nu\norm{w_{xx}(t)}^2\big)\Big)+e^{2\alpha t}\Big(\norm{w_{xt}(t)}^2+2\mu\norm{w_{xxt}(t)}^2+\frac{2}{\nu}E_3(w)(t)\Big)\\
	&\leq Ce^{2\alpha t}\Big(1+\tnorm{w_x(t)}^2\Big)\norm{w_{xx}(t)}^2+Ce^{2\alpha t}E_1(w)(t)\norm{w_x(t)}^2+Ce^{2\alpha t}\Big(E_1(w)(t)+E_3(w)(t)\Big).
	\end{align*}
	Integrate from $0$ to $t$ and then multiply the resulting inequality by $e^{-2\alpha t}$ with a use of Lemmas \ref{dlm1.2} and \ref{dlm1.4} to arrive at
	\begin{align*}
	\nu\norm{w_{xx}(t)}^2&+e^{-2\alpha t}\int_{0}^{t}e^{2\alpha s}\Big(\norm{w_{xt}(t)}^2+\mu\norm{w_{xxt}(t)}^2+\frac{2}{\nu}E_3(w)(s)\Big)\;ds\\
	&\leq Ce^{-2\alpha t}\Big(\big(\norm{w_0}^2+\mu\norm{w_{0x}}^2+\frac{\mu}{\nu}E_1(w)(0)\big)\Big)\\
	&\qquad\exp\Big(C\big((1+\mu)\norm{w_{0x}}^2+\mu\norm{w_{0xx}}^2+\frac{\mu}{\nu}E_1(w)(0)\big)\Big).
	\end{align*}
	This completes the proof.
\end{proof}
\subsection{Continuous dependence property}
Below, we show a continuous dependence property from which uniqueness follows.
\begin{lemma}\label{dlm1.6}
	For two different initial conditions $w_{10}$ and $w_{20}$ $\in H^1(0,1),$ the following continuous dependence property holds
	\begin{align*}
	\norm{z(t)}^2+\mu\norm{z_x(t)}^2+E_4(t)&\leq Ce^C\big(\norm{z_0}^2+\mu\norm{z_{0x}}^2+E_4(0)\big),
	\end{align*}
	where $z=w_1-w_2$, and $E_4(t)$ is same as in \eqref{dnex4}.
\end{lemma}
\begin{proof}	
	 Let $w_1$ and $w_2$ be two solutions of \eqref{deq1.7} with boundary conditions \eqref{deq1.13}, \eqref{deq1.14} and initial conditions $w_{10}$ and $w_{20}$,  and set $z=w_1-w_2$. Then, $z$ satisfies
	\begin{align}
	&z_t-\mu z_{xxt}-\nu z_{xx}+(1+w_d)z_x+w_1w_{1x}-w_2w_{2x}=0 \label{deq2.1},\\
	&z_x(0,t)=\frac{1}{\nu}\Big((1+c_0+w_d)z(0,t)+\frac{2}{9c_0}(w^3_1(0,t)-w^3_2(0,t))\Big)\label{deq2.2},\\
	&z_x(1,t)=-\frac{1}{\nu}\Big((1+c_1+w_d)z(0,t)+\frac{2}{9c_1}(w^3_1(1,t)-w^3_2(1,t))\Big)\label{deq2.3},\\
	& z(x,0)=w_{10}(x)-w_{20}(x)\label{deq2.4}.
	\end{align}
	In its weak formulation, seek $z\in H^1$ such that
	\begin{align}
	(z_t,v)+&\mu(z_{xt},v_x)+\nu(z_x,v_x)+(1+w_d)(z_x,v)+(w_1w_{1x}-w_2w_{2x},v)+\frac{\mu}{\nu}\bigg(\Big((1+c_0+w_d)z_t(0,t)\notag\\
	&\qquad+\frac{2}{9c_0}\frac{d}{dt}\big(w^3_1(0,t)-w^3_2(0,t)\big)\Big)v(0)+\Big((1+c_1+w_d)z_t(1,t)\notag\\
	&\qquad+\frac{2}{9c_1}\frac{d}{dt}\big(w^3_1(1,t)-w^3_2(1,t)\big)\Big)v(1)\bigg)
	+\Big((1+c_0+w_d)z(0,t)v(0)\notag\\
	&\qquad+(1+c_1+w_d)z(1,t)v(1)+\frac{2}{9c_0}\Big(\big(w^3_1(0,t)-w^3_2(0,t)\big)v(0)\notag\\
	&\hspace{2cm}
	+\frac{2}{9c_1}\big(w^3_1(1,t)-w^3_2(1,t\big)v(1)\Big)=0\label{deq2.5}.
	\end{align}
	Set $v=z$ in \eqref{deq2.5}, and bound the fourth and fifth terms on the left hand side, respectively, as
	\begin{equation*}
	(1+w_d)(z_x,z)=\frac{(1+w_d)}{2}(z^2(1,t)-z^2(0,t)),
	\end{equation*}
	and
	\begin{align*}
	(w_1w_{1x}-w_2w_{2x},z)
	&=(w_1z_x,z)+(zw_{2x},z)\\
	&\leq \norm{w_1(t)}_{L^\infty}\norm{z_x(t)}\norm{z(t)}+\norm{z(t)}_{L^\infty}\norm{w_{2x}(t)}\norm{z(t)}\\ &\leq
	\sqrt{2}\tnorm{w_1(t)}\norm{z_x(t)}\norm{z(t)}+(\lvert z(0,t)\rvert+\norm{z_x(t)})\norm{w_{2x}(t)}\norm{z(t)}\\
	&\leq  \frac{\nu}{4}\norm{z_x(t)}^2+\frac{1+c_0+w_d}{2}z^2(0,t)+C(\tnorm{w_{1}(t)}^2+\tnorm{w_{2}(t)}^2)\norm{z(t)}^2.
	\end{align*}
	Now to bound the other terms on the left hand side of \eqref{deq2.5}, we rewrite the following terms as for $i=0,\hspace{0.1cm}1$
	\begin{align*}
	\big(w^3_1(i,t)-w^3_2(i,t)\big)z(i,t)&=z^2(i,t)\big(w^2_1(i,t)+w_1(i,t)w_2(i,t)+w_2^2(i,t)\big)\\
	&\quad \geq z^2(i,t)\big(w^2_1(i,t)-|w_1(i,t)||w_2(i,t)|+w_2^2(i,t)\big) \\
	&\quad \geq \frac{1}{2}z^2(i,t)(w_1^2(i,t)+w_2^2(i,t))\geq 0,
	\end{align*}
	and 
	\begin{align*}
	\frac{d}{dt}\big(w^3_1(i,t)-w^3_2(i,t)\big)z(i,t)&=\frac{3}{2}\frac{d}{dt}\Big(w_1^2(i,t)z^2(i,t)\Big)-3w_1(i,t)w_{1t}(i,t)z^2(i,t)\\
	&\quad+3z^2(i,t)\big(w_1(i,t)+w_2(i,t)\big)w_{2t}(i,t)\\
	&\leq \frac{3}{2}\frac{d}{dt}\Big(w_1^2(i,t)z^2(i,t)\Big)+Cz^2(i,t)\Big(w_1^2(i,t)+w_{1t}^2(i,t)\\
	&\qquad+w_2^2(i,t)+w_{2t}^2(i,t)\Big).
	\end{align*}
	Therefore,  from \eqref{deq2.5}, we arrive at 
	\begin{align*}
	\frac{d}{dt}\Bigg(\norm{z(t)}^2+\mu\norm{z_x(t)}^2&+\frac{\mu}{\nu}\Big((1+c_0+w_d)z^2(0,t)+(1+c_1+w_d)z^2(1,t)+\frac{1}{3c_0}w_1^2(0,t)z^2(0,t)
	\\
	&\quad +\frac{1}{3c_1}w_1^2(1,t)z^2(1,t)\Big)\Bigg)+\nu\norm{z_x(t)}^2\\
	&\leq C(\tnorm{w_{1}}^2+\tnorm{w_{2}}^2)\norm{z}^2
	+C\frac{\mu}{\nu}z^2(0,t)\Big(w_1^2(0,t)+w_{1t}^2(0,t)+w_2^2(0,t)\\
	&\qquad+w_{2t}^2(0,t)\Big)+C\frac{\mu}{\nu}z^2(1,t)\Big(w_1^2(1,t)+w_{1t}^2(1,t)+w_2^(1,t)+w_{2t}^2(1,t)\Big).
	\end{align*}
	Setting
	\begin{align}
	E_4(t)=&\frac{\mu}{\nu}\Big((1+c_0+w_d)z^2(0,t)+(1+c_1+w_d)z^2(1,t)\notag\\
	&\qquad+\frac{1}{3c_0}w_1^2(0,t)z^2(0,t) +\frac{1}{3c_1}w_1^2(1,t)z^2(1,t)\Big)\label{dnex4},
	\end{align} 
	we obtain
	\begin{align*}
	\frac{d}{dt}\Big(\norm{z(t)}^2+\mu\norm{z_x(t)}^2+E_4(t)\Big)+\nu \norm{z_x(t)}^2
	&\leq C\Big(\norm{z(t)}^2+\mu\norm{z_x(t)}^2+E_4(t)\Big)\Big(w_{1t}^2(0,t)+w_{2t}^2(0,t)\\
	&\qquad+w_{1t}^2(1,t)+w_{2t}^2(1,t)+\tnorm{w_{1}(t)}^2+\tnorm{w_{2}(t)}^2\Big).
	\end{align*}
	Applying Gronwall's inequality to the above inequality yields
	\begin{align*}
	\norm{z(t)}^2+\mu\norm{z_x(t)}^2+E_4(t)&\leq\big(\norm{z_0}^2+\mu\norm{z_{0x}}^2+E_4(0)\big)\exp\Bigg(C\int_{0}^{t}\Big(w_{1t}^2(0,s)+w_{2t}^2(0,s)
	\\
	&\qquad+w_{1t}^2(1,s)+w_{2t}^2(1,s)+\tnorm{w_{1}(t)}^2+\tnorm{w_{2}(t)}^2\Big)\; ds\Bigg).
	\end{align*}
	A use of Lemmas \ref{dlm1.2}-\ref{dlm1.4} gives the desired result.
\end{proof}
As a consequence, when $w_{10}=w_{20},$ it follows that $w_1(t)=w_2(t)$ for all $t>0$. Hence, the solution is unique.
 \section{Finite element approximation}
In this section, we discuss semidiscrete Galerkin approximation keeping the time variable continuous. Moreover, optimal error estimates for the state variable and superconvergence results for feedback controllers are established.

For any positive integer $N,$ let $\Pi =\left\{0 = x_0 < x_1 <\cdots < x_{N} = 1\right\}$ be a partition of $\overline {I}$ into subintervals  $I_j = (x_{j-1} , x_j ),\hspace{0.1cm} 1\leq j \leq N $ with
$h_j = x_j-x_{j-1}$ and mesh parameter $h = \max\limits_{1\leq j\leq N }h_j$. We define a finite dimensional subspace $V_h$ of $H^1$ as follows
$$ V_h=\left\{v_h\in C^0\big(\overline{I} \big):\hspace{0.1cm} v_h\Big|_{I_j} \in \mathcal{P}_{1}(I_j) \quad 1\leq j\leq N\right\},$$
where $\mathcal{P}_{1}(I_j)$ is the set of linear polynomials in $I_j$.

Now, the corresponding semidiscrete formulation for the problem \eqref{deqx2.1}-\eqref{deqx2.4} is to seek $w_h(t)\in V_h,$ $t>0$ such that
\begin{align}
(w_{ht},\chi)+&\mu(w_{hxt},\chi_x)+\nu(w_{hx},\chi_x)+w_d(w_{hx},\chi)+(w_hw_{hx},\chi)
+\Big(\big((1+c_0+w_d)w_h(0,t)\notag\\
&+\frac{2}{9c_0}w_h^3(0,t)\big)\chi(0)
+\big((1+c_1+w_d)w_h(1,t)+\frac{2}{9c_1}w_h^3(1,t)\big)\chi(1)\Big)\notag\\
&+\frac{\mu}{\nu}\Big(\big((c_0+1+w_d)w_{ht}(0,t)+\frac{2}{3c_0}w_h^2(0,t)w_{ht}(0,t)\big)\chi(0)+\big((c_1+1+w_d)w_{ht}(1,t)\notag\\
&+\frac{2}{3c_1}w_h^2(1,t)w_{ht}(1,t)\big)\chi(1)\Big)
=0  \qquad \forall~ \chi \in V_h
\label{deq4.3}
\end{align}
with $w_h(x,0)=w_{0h}(x),$ an approximation of $w_0$. For our analysis, we assume that $w_{0h}$ is the $H^1$ projection of $w_0$ onto $V_h$.\\
Now since $V_h$ is finite dimensional, the semidiscrete problem \eqref{deq4.3} leads to a system of nonlinear ODEs. Then an appeal to the Picard's theorem yields the existence of a unique solution $w_h(t)$ in $t\in (0,t^*)$ for some $t>0$. Since from Lemma \ref{dlm2.1}, $w_h(t)$ is bounded for all $t>0,$ using a continuation argument, the global existence of $w_h(t)$ is established.\\ 
Below, we state four Lemmas for the semidiscrete problem \eqref{deq1.7}-\eqref{deq1.10}, which imply global stabilization result for the semidiscrete solution.
\begin{lemma}\label{dlm2.1}
	Let $w_0\in H^1(0,1)$. With $\alpha$ as in \eqref{denq1.1}, there holds
	\begin{align*}
	\norm{w_h(t)}^2+&\mu\norm{w_{hxt}(t)}^2+\frac{\mu}{\nu}E_{1h}(w_h)(t)+\beta e^{-2\alpha t}\int_{0}^{t} e^{2\alpha s}\tnorm{w_h(t)}^2 ds\\
	&\leq Ce^{-2\alpha t}\big(\norm{w_{0h}}^2+\mu\norm{w_{0hx}}^2+\frac{\mu}{\nu}E_{1h}(w_h)(0)\big),
	\end{align*}
	where 
	\begin{equation*}
	E_{1h}(w_h)(t)=\Big((c_0+1+w_d)w_h^2(0,t)+(c_1+1+w_d)w_h^2(1,t)+\frac{1}{3c_0}w_h^4(0,t)+\frac{1}{3c_1}w_h^4(1,t)\Big),
	\end{equation*}
	and $\beta$ is the same as in \eqref{deqn1.1}.
\end{lemma}
\begin{proof}
For the proof we can proceed as in continuous case.
\end{proof}
One dimensional {\it{discrete Laplacian}} $\Delta_h:V_h\longrightarrow V_h$ is defined by 
\begin{align}\label{deqx3.9}
(-\Delta_hv_h,w_h)=(v_{hx},w_{hx})+v_{hx}(0)w_h(0)-v_{hx}(1)w_h(1) \qquad \forall~ v_h, w_h\in V_h.
\end{align}
The semidiscrete version of the control problem \eqref{deq1.7}-\eqref{deq1.10} satisfies
\begin{align}
&w_{ht}-\mu \Delta_hw_{ht}+\nu \Delta_hw_h+(1+w_d)w_{hx}+w_hw_{hx}=0 \label{deqx2.1},\\
&w_{hx}(0,t)=:v_{0h}(t)=:\frac{1}{\nu}\Big((1+c_0+w_d)w_h(0,t)+\frac{2}{9c_0}w_h^3(0,t)\Big)  \label{deqx2.2},\\
&w_{hx}(1,t)=:v_{1h}(t)=:-\frac{1}{\nu}\Big((1+c_1+w_d)w_h(1,t)+\frac{2}{9c_1}w_h^3(1,t)\Big)  \label{deqx2.3},\\
&w_h(x,0)=w_{0h}(x)\hspace{0.1cm}(\text{say}) \label{deqx2.4},
\end{align}
where 
following estimates hold:
\begin{equation}\label{derror-u0}
\|w_0- w_{0h}\|_j \leq C h^{2-j} \|w_0\|_2, \quad j=0,1.
\end{equation}
Using \eqref{derror-u0}, we can show that $\norm{w_{0h}}\leq \norm{w_{0}}$ and $\norm{w_{0hx}}\leq \norm{w_{0x}}$. For showing the bound of $\norm{\Delta_h w_{0h}},$ we rewrite
\begin{align*}
\Big(-\Delta_h w_{0h},\phi_{h}\Big)&=(w_{0hx},\phi_{hx})+w_{0hx}(0)\phi_{h}(0)-w_{0hx}(1)\phi_{h}(1)\\
&=(-w_{0xx},\phi_{h})-\Big((w_{0x}-w_{0hx},\phi_{hx})+(w_{0x}(0)-w_{0hx}(0))\phi_{h}(0)\\
&\qquad-(w_{0x}(1)-w_{0hx}(1))\phi_{h}(1)\Big). 
\end{align*}
Choose $\tilde w_h(0)=w_{0h}$ so that from Lemma \ref{dlm3.1}, we obtain the bound of
$|w_{0x}(0)-w_{0hx}(0)|$ and $|w_{0x}(1)-w_{0hx}(1)|$. Now a use of inverse inequality yields $\norm{\Delta_h w_{0h}}\leq C\norm{w_{0xx}}$ easily.
\begin{lemma}\label{dlm2.2}
	Let $w_0\in H^2(0,1)$. Then, there holds
	\begin{align*}
	\Big(\norm{w_{hx}(t)}^2&+\mu\norm{\Delta_hw_h(t)}^2+\frac{1}{\nu}E_{2h}(w_h)(t)\Big)+\beta e^{-2\alpha t}\int_{0}^{t} e^{2\alpha s}\norm{\Delta_hw_h(s)}^2 ds\\
	&\leq C(1+\mu)e^{-2\alpha t}\big(\norm{w_{0}}^2_1\big)\exp\Big(C(1+\mu)\norm{w_0}^2_1\Big),
	\end{align*}
	where
	$$E_{2h}(w_h)(t)=\Big((c_0+1+w_d)w_h^2(0,t)+(c_1+1+w_d)w_h^2(1,t)
	+\frac{1}{9c_0}w_h^4(0,t)+\frac{1}{9c_1}w_h^4(1,t)\Big).$$
\end{lemma}
\begin{lemma}\label{dlm2.3}
	Let $w_0\in H^2(0,1)$. Then, there holds
	\begin{align*}
	\nu\big(&\norm{w_{hx}(t)}^2+E_{2h}(w_h)(t)\big)+e^{-2\alpha t}\int_{0}^{t}e^{2\alpha s}\Big(\norm{w_{ht}}^2+\mu\norm{w_{hxt}}^2+\frac{\mu}{\nu}E_{3h}(w_h)(s)\Big) ds\\
	&\qquad\leq C(1+\mu)e^{-2\alpha t}\big(\norm{w_{0}}^2_2\big)\exp\Big(C(1+\mu)\norm{w_0}^2_1\Big),
	\end{align*}
	where
	\begin{align*}
	E_{3h}(w_h)(t)&=\Big((1+c_0+w_d)w_{ht}^2(0,t)+(1+c_1+w_d)w_{ht}^2(1,t)+\frac{2}{3c_0}w_h^2(0,t)w_{ht}^2(0,t)\\
	&\qquad+\frac{2}{3c_1}w_h^2(1,t)w_{ht}^2(1,t)\Big).
	\end{align*}
\end{lemma}
\begin{lemma}\label{dlm2.4}
	Let $w_0\in H^2(0,1)$. Then, there holds
	\begin{align*}
	\Big(\norm{w_{ht}(t)}^2&+\mu\norm{w_{hxt}(t)}^2+\frac{\mu}{\nu}E_{3h}(w_h)(t)\Big)+e^{-2\alpha t}\int_{0}^{t}e^{2\alpha s}\Big(\nu\norm{w_{hxt}(s)}^2+E_{3h}(w_h)(s)\Big) ds\\
	&\leq
C(1+\mu)e^{-2\alpha t}\big(\norm{w_{0}}^2_1\big)\exp\Big(C(1+\mu)\norm{w_0}^2_1\Big),
	\end{align*}
	where $E_{3h}(w_h)(t)$ is as in previous Lemma \ref{dlm2.3}.
\end{lemma}
\begin{remark}\label{drm2.1}
The proofs of the above Lemmas \ref{dlm2.2}-\ref{dlm2.4} follows in a similar fashion as in continuous case. Also for $\alpha=0,$ all results in these lemmas hold.
\end{remark}
 \subsection{Error estimates}
 To bound the error, we first
 introduce an auxiliary projection $\tilde w_h\in V_h$ of $w$ through the following form
 \begin{equation}\label{deq4.4}
 	(w_x-\tilde{w}_{hx},\chi_x)+\lambda (w-\tilde w_h,\chi)=0 \quad \chi\in V_h, 
 \end{equation}
 where $\lambda$ is some fixed positive number. For a given $w\in H^1,$ the existence of a unique $\tilde w_h$ follows from the Lax-Milgram Lemma. 
 Let $\eta:=w-\tilde w_h$ be the error involved in the auxiliary projection. Then, the following standard error estimates hold 
 \begin{align}\label{deq4.5}
 	\norm{\eta(t)}_j\leq C h^{\min(2,m)-j}\;\norm{w(t)}_m, \;\text{and} \;\norm{\eta _t(t)}_j\leq C h^{\min(2,m)-j}\norm{w_t(t)}_m,\;\;
 	j=0,1 \;\mbox{ and }\; m=1,2.
 \end{align}
 and
 \begin{align}\label{deqx5.4}
 	\norm{\eta(t)}_{L^\infty}\leq Ch^2\norm{w(t)}_{2,\infty}.
 \end{align}
 For a proof, we refer to Thom\'{e}e \cite{thomee}.
 In addition, for proving optimal error estimates, we need the following estimates of $\eta$ and $\eta_t$ at the boundary points $x=0,1$  whose proof can be found out in \cite{skakp1} and \cite{pani}. 
 \begin{lemma}\label{dlm3.1}
 	For $x=0,1,$ there holds
 	\begin{align*}
 		|\eta(x,t)|\leq C h^2\norm{w(t)}_2 \text {and} \qquad|\eta_t(x,t)|\leq C h^2\norm{w_t(t)}_2.
 	\end{align*}
 \end{lemma}
 Using elliptic projection, write $$e:=w-w_h=(w-\tilde w_h)-(w_h-\tilde w_h)=:\eta-\theta.$$
 Choose $\tilde w_h(0)=w_{0h}$ so that $\theta(0)=0$.\\
 Since estimates of $\eta$ are known, it is enough to estimate $\theta$.
 Subtracting \eqref{deq4.3} from \eqref{deqx1.11} and using \eqref{deq4.4}, we arrive at 
 \begin{align}
 	(\theta_t,\chi)+&\mu(\theta_{xt},\chi_x)+\nu(\theta_x,\chi_x)+\sum_{i=0}^{1}(1+c_i+w_d)\theta(i,t)\chi(i)+\frac{\mu}{\nu}\sum_{i=0}^{1}(1+c_i+w_d)\theta_t(i,t)\chi(i)\notag\\
 	&=\Big((\eta_t,\chi)-\mu\lambda(\eta_t,\chi)-\nu\lambda(\eta,\chi)\Big)+(1+w_d)(\eta_x-\theta_x,\chi)+(ww_x-w_hw_{hx},\chi)\notag\\
 	&\quad+\sum_{i=0}^{1}\Big((1+c_i+w_d)\eta(i,t)\chi(i)+\frac{2}{9c_i}\big(w^3(i,t)-w^3_h(i,t)\big)\chi(i)\Big)\notag\\
 	&\quad+\frac{\mu}{\nu}\sum_{i=0}^{1}\Big((1+c_i+w_d)\eta_t(i,t)\chi(i)+\frac{2}{9c_i}\frac{d}{dt}\big(w^3(i,t)-w_h^3(i,t)\big)\chi(i)\Big)
 	\label{deq4.7},
 \end{align}
 where $w^3(i,t)-w_h^3(i,t)$ for $i=0,\hspace{0.1cm}1$ can be rewritten as
 \begin{align*}
 	w^3(i,t)-w_h^3(i,t)&=\eta^3(i,t)-\theta^3(i,t)+3w(i,t)\eta(i,t)\big(w(i,t)-\eta(i,t)\big)\\
 	&\qquad-3w_h(i,t)\theta(i,t)\big(w_h(i,t)-\theta(i,t)\big).
 \end{align*}
 \begin{lemma}\label{dlm3.2}
 	Assume that $w_0\in H^2(0,1)$. Then, there exists a positive constant $C$ independent of $h$ such that 
 	\begin{align*}
 	\Big(\norm{\theta(t)}^2&+\mu\norm{\theta_x(t)}^2+\frac{1}{9}\frac{\mu}{\nu}E_1(\theta)(t)\Big)+ \frac{\beta_1}{2}e^{-2\alpha t}\int_{0}^{t}e^{2\alpha s}\Big(\norm{\theta_x(s)}^2+E_1(\theta)(s)\Big) \;ds\\
 	&\leq  C\frac{1}{\mu}(\norm{w_0}_2) (1+\mu)h^4e^{-2\alpha t} \exp\Big(C\norm{w_0}_2\Big),
 	\end{align*}
 		where $\beta_1=\min \Bigg\{(\frac{3\nu}{2}-2\alpha(\mu+1)), \Big(1-2\alpha\big(\frac{2\mu}{\nu}+1\big)\Big)\Bigg\}>0$.
 \end{lemma}
 \begin{proof}
 	Set $\chi=\theta$ in \eqref{deq4.7} to obtain
 	\begin{align}
 	\frac{1}{2}\frac{d}{dt}\big(&\norm{\theta(t)}^2+\mu\norm{\theta_x(t)}^2)+\nu \norm{\theta_x(t)}^2+\sum_{i=0}^{1}(c_i+(1+w_d))\theta^2(i,t)+\frac{\mu}{2\nu}\frac{d}{dt}\Big(\sum_{i=0}^{1}(1+c_i+w_d)\theta^2(i,t)\Big)\notag\\
 	&=\Big((\eta_t,\theta)-\mu\lambda(\eta_t,\theta)-\nu\lambda(\eta,\theta)\Big)+(1+w_d)(\eta_x-\theta_x,\theta)+\Big(w(\eta_x-\theta_x)+(\eta-\theta)w_{hx},\theta\Big)\notag\\
 	&\qquad
 	+\sum_{i=0}^{1}\Big((1+c_i+w_d)\eta(i,t)\theta(i,t)+\frac{2}{9c_i}\big(w^3(i,t)-w_h^3(i,t)\big)\theta(i,t)\Big)\notag\\
 	&\qquad+\frac{\mu}{\nu}\sum_{i=0}^{1}\Bigg((1+c_i+w_d)\eta_t(i,t)\theta(i,t)+\frac{2}{9c_i}
 	\frac{d}{dt}\big(w^3(i,t)-w_h^3(i,t)\big)\theta(i,t)\Bigg)= \sum_{j=1}^{5}I_j(\theta)\label{deq4.8},
 	\end{align}
where $I_4(\theta)$ and $I_5(\theta)$ are last two summation term respectively.
 	The first term on the right hand side of \eqref{deq4.8} is bounded using the Cauchy-Schwarz inequality and the Young's inequality in
 		\begin{align*}
 	I_1(\theta)=(\eta_t,\theta) -\mu\lambda(\eta_t,\theta)-\nu\lambda(\eta,\theta)&\leq \frac{\epsilon}{4}\norm{\theta(t)}^2+C(1+\mu^2)\norm{\eta_t(t)}^2+C\norm{\eta(t)}^2,
 	\end{align*}
 	where constant $\epsilon>0$ we choose later.
 	For	the second term on the right hand side of \eqref{deq4.8},  integration by parts, the Cauchy-Schwarz inequality, and  Young's inequality yield
 	\begin{align*}
 	I_2(\theta)=(1+w_d)(\eta_x-\theta_x,\theta)&=-(1+w_d)\Big((\eta,\theta_x)+\eta(1,t)\theta(1,t)-\eta(0,t)\theta(0,t)\Big)\\
 	&\qquad -\frac{(1+w_d)}{2}\Big(\theta^2(1,t)-\theta^2(0,t)\Big)\\
 	&\leq \frac{\nu}{8}\norm{\theta_x(t)}^2+\frac{c_0}{8}\theta^2(0,t)+\frac{c_1}{8}\theta^2(1,t)+C\norm{\eta(t)}^2
 	\\
 	&\qquad	+C(\eta^2(0,t)+\eta^2(1,t)) -\frac{(1+w_d)}{2}\Big(\theta^2(1,t)-\theta^2(0,t)\Big).
 	\end{align*}
 	For the third term, we note that
 		\begin{align*}
 	I_3(\theta)=\big(w(\eta_x-\theta_x)+(\eta-\theta)w_{hx},\theta)&=-(w\eta,\theta_x)-(w_x\eta,\theta)-(w\theta_x,\theta)+\big((\eta-\theta)w_{hx},\theta\big)\\
 	&\quad+w(1,t)\eta(1,t)\theta(1,t)-w(0,t)\eta(0,t)\theta(0,t)
 	\\
 	&\leq \frac{\nu}{8}\norm{\theta_x(t)}^2+C\norm{w(t)}^2_{L^\infty}\norm{\eta(t)}^2+C\norm{\eta(t)}^2 +\frac{c_0}{8}\theta^2(0,t)\\
 	&\quad+\frac{c_1}{8}\theta^2(1,t)+C\big(\norm{w_x(t)}^2_{L^\infty}+\norm{w_{hx}(t)}^2_{L^\infty}+\norm{w(t)}^2_{L^\infty}\big)\norm{\theta(t)}^2\\
 	&\quad+\frac{\epsilon}{4}\norm{\theta(t)}^2+C(w^2(0,t)\eta^2(0,t)+w^2(1,t)\eta^2(1,t)).
 	\end{align*}
 	First subterms of the fourth and fifth term on the right hand side of \eqref{deq4.8} are bounded by
 	\begin{align*}
 	&(1+c_0+w_d)\eta(0,t)\theta(0,t)+(1+c_1+w_d)\eta(1,t)\theta(1,t)\\
 	&\qquad+\frac{\mu}{\nu}\Big((1+c_0+w_d)\eta_t(0,t)\theta(0,t)+(1+c_1+w_d)\eta_t(1,t)\theta(1,t)\Big)\\
 	&\leq \frac{c_0}{8}\theta^2(0,t)+ \frac{c_1}{8}\theta^2(1,t)+C\Big(\eta^2(0,t)+\eta^2(1,t)\\
 	&\qquad+\mu^2\big(\eta_t^2(0,t)+\eta_t^2(1,t)\big)\Big).
 	\end{align*}
 	For second subterm of the fourth term on the right hand side, 
 	we note that for $i=0,\hspace{0.1cm}1$
 	\begin{align*}
 	\frac{2}{9c_i}\Big(w^3(i,t)-w_h^3(i,t)\Big)\theta(i,t)&=-\frac{2}{9c_i}\theta^4(i,t)-\frac{2}{3c_i}w_h^2(i,t)\theta^2(i,t)-\frac{2}{9c_i}\eta^3(i,t)\theta(i,t)\\
 	&\quad+\frac{2}{3c_i}\Big(w^2(i,t)\eta(i,t)-w(i,t)\eta^2(i,t)+w_h(i,t)\theta^2(i,t)\Big)\theta(i,t).
 	\end{align*}
 	Using Young's inequality,  implies that for $i=0,\hspace{0.1cm}1$
 	\begin{align*}
 	\frac{2}{9c_i}\eta^3(i,t)\theta(i,t)\leq \frac{2}{9c_i}\frac{1}{16}\theta^4(i,t)+C\eta^4(i,t),
 	\end{align*}
 	\begin{align*}
 	\frac{2}{3c_i}w^2(i,t)\eta(i,t)\theta(i,t)\leq \frac{c_i}{8}\theta^2(i,t)+Cw^4(i,t)\eta^2(i,t)
 	\end{align*}
 	and
 	\begin{align*}
 	\frac{2}{3c_i}w(i,t)\eta^2(i,t)\theta(i,t)\leq \frac{2}{9c_i}\frac{1}{16}\theta^4(i,t)+C(w(i,t)\eta^2(i,t))^\frac{4}{3}.
 	\end{align*}
 	Again, a use of Young's inequality yields
 	\begin{align*}
 	\frac{2}{3c_i}w_h(i,t)\theta^3(i,t)\leq \frac{2}{3c_i}12w_h^2(i,t)\theta^2(i,t)+\frac{2}{9c_i}\frac{1}{16}\theta^4(i,t).
 	\end{align*}
 	Hence, the contribution of the second subterm of the fourth term on the right hand side of \eqref{deq4.8} after applying Lemmas \ref{dlm1.2}, \ref{dlm1.4} and \ref{dlm3.1}, can be bounded as
 	\begin{align*}
 	\sum_{i=0}^{1}&\frac{2}{9c_i}\Big(w^3(i,t)-w_h^3(i,t)\Big)\theta(i,t)\\
 	&\leq -\frac{2}{9c_0}\frac{13}{16}\theta^4(0,t)+\frac{c_0}{8}\theta^2(0,t)+C\eta^2(0,t)
 	\\
 	&\quad
 	-\frac{2}{9c_1}\frac{13}{16}\theta^4(1,t)+\frac{c_1}{8}\theta^2(1,t)+C\eta^2(1,t)+8\sum_{i=0}^{1}\frac{1}{c_i}w_h^2(i,t)\theta^2(i,t).
 	\end{align*}
 Expanding the second subterm of the fifth term, we note that for $i=0,\hspace{0.1cm}1$
 	\begin{align*}
 	\frac{2}{9c_i}\frac{\mu}{\nu}\frac{d}{dt}\big(\eta^3(i,t)\big)\theta(i,t)\leq\frac{\mu}{\nu}C\theta^2(i,t)\eta^2(i,t)+C\frac{\mu}{\nu}\eta^2(i,t)\eta_t^2(i,t),
 	\end{align*}  
 	\begin{align*}
 	\frac{2}{9c_i}\frac{\mu}{\nu}\frac{d}{dt}\big(\theta^3(i,t)\big)\theta(i,t)=-\frac{1}{6c_i}\frac{\mu}{\nu}\frac{d}{dt}\big(\theta^4(i,t)\big),
 	\end{align*}
 	and using Lemmas \ref{dlm1.2} and \ref{dlm1.4}, we obtain
 	\begin{align*}
 	\frac{2}{3c_i}\frac{\mu}{\nu}\frac{d}{dt}\Big(w^2(i,t)\eta(i,t)\Big)\theta(i,t)\leq \frac{\mu}{\nu}C\theta^2(i,t)w^2(i,t)+C\frac{\mu}{\nu}\eta^2(i,t)+C\frac{\mu}{\nu}\eta_t^2(i,t).
 	\end{align*}
 	Also, it holds that
 	\begin{align*}
 	-\frac{2}{3c_i}\frac{\mu}{\nu}\frac{d}{dt}\Big(w(i,t)\eta^2(i,t)\Big)\theta(i,t)&\leq \frac{\mu}{\nu}C\theta^2(i,t)\big(w^2(i,t)+w_t^2(i,t)\big)+C\frac{\mu}{\nu}\eta^4(i,t)\\
 	&\qquad+C\frac{\mu}{\nu}\eta^2(i,t)\eta_t^2(i,t).
 	\end{align*}
 Rewrite and use the Young's inequality to obtain
 	\begin{align*}
 	-\frac{2}{3c_i}\frac{\mu}{\nu}\frac{d}{dt}\Big(w_h^2(i,t)\theta(i,t)\Big)\theta(i,t)&\leq-\frac{1}{3c_i}\frac{\mu}{\nu}\frac{d}{dt}\Big(w_h^2(i,t)\theta^2(i,t)\Big)\\
 	&\qquad+C\theta^2(i,t)\Big(w_h^2(i,t)+\mu^2 w_{ht}^2(i,t)\Big).
 	\end{align*}
 	Similarly,
 	\begin{align*}
 	\frac{2}{3c_i}\frac{\mu}{\nu}\frac{d}{dt}\Big(w_h(i,t)\theta^2(i,t)\Big)\theta(i,t)&\leq\frac{4}{9c_i}\frac{\mu}{\nu}\frac{d}{dt}\big(w_h(i,t)\theta^3(i,t)\big)+\frac{2}{9c_i}\frac{1}{16}\theta^4(i,t)\\
 	&\qquad+C\frac{\mu^2}{\nu^2}w_{ht}^2(i,t)\theta^2(i,t).
 	\end{align*}
 	Hence, from \eqref{deq4.8}, we arrive using Lemmas \ref{dlm1.2}, \ref{dlm1.4}, \ref{dlm2.2}, \ref{dlm2.4} and \ref{dlm3.1} at
 		\begin{align*}
 	\frac{d}{dt}\big(\norm{\theta(t)}^2&+\mu\norm{\theta_x(t)}^2)+\frac{3\nu}{2}\norm{\theta_x(t)}^2+\sum_{i=0}^{1}\frac{4}{3c_i}w_h^2(i,t)\theta^2(i,t)+E_1(\theta)(t)\\
 	&\qquad+\frac{\mu}{\nu}\frac{d}{dt}\Big(E_1(\theta)(t)+\sum_{i=0}^{1}\frac{2}{3c_i}w_h^2(i,t)\theta^2(i,t)\Big)\\
 	&\leq C((1+\mu^2)\norm{\eta_t(t)}^2+\norm{\eta(t)}^2)+C\Big(\tnorm{w(t)}^2+\norm{w_{xx}(t)}^2+\norm{\Delta_hw_h(t)}^2+w_x^2(0,t)\\
 	&\qquad+w_{hx}^2(0,t)\Big)\norm{\theta}^2+\epsilon\norm{\theta}^2+C(1+\mu)\Big(\sum_{i=0}^{1}\eta^2(i,t)\Big)+16\sum_{i=0}^{1}\frac{1}{c_i}w_h^2(i,t)\theta^2(i,t) \\
 	&\qquad+\frac{\mu}{\nu}\frac{d}{dt}\Big(\sum_{i=0}^{1}\frac{8}{9c_i}w_h(i,t)\theta^3(i,t)\Big)+C(\mu+\mu^2)\Big(\sum_{i=0}^{1}\eta_t^2(i,t)\Big)
 	\\
 	&\qquad+C\sum_{i=0}^{1}\frac{\mu}{\nu}\theta^2(i,t)\Big(w^2(i,t)+w_t^2(i,t)+w_h^2(i,t)+\mu w_{ht}^2(i,t)\Big).
 	\end{align*}
 	Multiply the above inequality  by $e^{2\alpha t}$. Use  Poincar\'e-Wirtinger's  inequality $$\norm{\theta(t)}^2\leq \theta^2(0,t)+\theta^2(1,t)+\norm{\theta_x(t)}^2
 	\leq E_1(\theta)(t)+\norm{\theta_x(t)}^2
 	$$ with 
 	\begin{align*}
 	2\alpha\frac{\mu}{\nu}\sum_{i=0}^{1}\frac{8}{9c_i}w_h(i,t)\theta^3(i,t)&\geq 
 	-2\alpha\frac{\mu}{\nu}\sum_{i=0}^{1}\frac{2}{3c_i}w_h^2(i,t)\theta^2(i,t)-2\alpha\frac{\mu}{\nu}E_1(\theta)(t)
 	\end{align*}
 This	yields 
  	\begin{align*}
 \frac{d}{dt}&\Bigg(e^{2\alpha t}\Big(\norm{\theta(t)}^2+\mu\norm{\theta_x(t)}^2+\frac{\mu}{\nu}\big(E_1(\theta)(t)+\sum_{i=0}^{1}\frac{2}{3c_i}w_h^2(i,t)\theta^2(i,t)\big)\Big)\Bigg)\\
 &\qquad+e^{2\alpha t}\sum_{i=0}^{1}\frac{4}{3c_i}w_h^2(i,t)\theta^2(i,t)+e^{2\alpha t}\Big(1-2\alpha\big(\frac{2\mu}{\nu}+1\big)\Big)E_1(\theta)(t)\\
 &\qquad+\Big(\frac{3\nu}{2}-2\alpha(1+\mu)\Big) e^{2\alpha t}\norm{\theta_x(t)}^2\\
 &\leq Ce^{2\alpha t}\big((1+\mu^2)\norm{\eta_t}^2+\norm{w(t)}^2_1\norm{\eta(t)}^2\big)+Ce^{2\alpha t}\Big(\phi(t)+w_x^2(0,t)+w_{hx}^2(0,t)\Big)\norm{\theta}^2
 \\
 &\qquad+C\frac{\mu}{\nu}\sum_{i=0}^{1}\theta^2(i,t)\Big(w^2(i,t)+w_t^2(i,t)+w_h^2(i,t)+\mu w_{ht}^2(i,t)\Big)+\epsilon e^{2\alpha t}\Big(E_1(\theta)(t)+\norm{\theta_x(t)}^2\Big)\\
 &\qquad+\frac{\mu}{\nu}\frac{d}{dt}\Bigg(e^{2\alpha t}\sum_{i=0}^{1}\Big(\frac{8}{9c_i}w_h(i,t)\theta^3(i,t)\Big)\Bigg)+e^{2\alpha t}(4\alpha\frac{\mu}{\nu}+24)\sum_{i=0}^{1}\frac{2}{3c_i}w_h^2(i,t)\theta^2(i,t)\\
 &\qquad+C(1+\mu)e^{2\alpha t}\big(\sum_{i=0}^{1}\eta^2(i,t)\big)+C(\mu+\mu^2)e^{2\alpha t}\Big(\sum_{i=0}^{1}\eta_t^2(i,t)\Big),
 \end{align*}
 	where $\phi(t)=\tnorm{w(t)}^2+\norm{w_{xx}(t)}^2+\norm{\Delta_hw_h(t)}^2$.
 	Now integrate from $0$ to $t$ and choose $\epsilon=\frac{\beta_1}{2}$ with 
 	 	\begin{align*}
 	2\alpha\frac{\mu}{\nu}\frac{8}{9c_i}w_h(i,t)\theta^3(i,t)&\geq -2\alpha\frac{\mu}{\nu}\frac{2}{3c_i}w_h^2(i,t)\theta^2(i,t)-2\alpha\frac{\mu}{\nu}\frac{8}{27c_i}\theta^4(i,t)
 	, \quad i=0,1.
 	\end{align*}
 	to find that  
 	\begin{align}
 	e^{2\alpha t}\Bigg(\norm{\theta(t)}^2+&\mu\norm{\theta_x(t)}^2+\frac{1}{9}\frac{\mu}{\nu}E_1(\theta)(t)\Bigg)+\frac{\beta_1}{2}\int_{0}^{t}e^{2\alpha s}\Bigg(\norm{\theta_x(t)}^2+E_1(\theta)(s)\Bigg)\;ds\notag\\
 	&\qquad+\int_{0}^{t}e^{2\alpha s}\Big(\sum_{i=0}^{1}\frac{4}{3c_i}w_h^2(i,s)\theta^2(i,s)\Big) ds\\
 	&\leq Ch^4\int_{0}^{t}e^{2\alpha s}\Big((1+\mu+\mu^2)\norm{w_t(t)}^2_2+(1+\mu)\norm{w(t)}_2^2\Big)\; ds\notag\\
 	&\qquad+C\int_{0}^{t}e^{2\alpha s}\Big(\phi(t)+w_x^2(0,t)+w_{hx}^2(0,t)\Big)\norm{\theta(t)}^2 ds\notag\\
 	&\qquad+C(\frac{\mu}{\nu}+1)\int_{0}^{t}e^{2\alpha s}E_1(\theta)(s)\psi(s)\;ds\label{dex1},
 	\end{align}
 	where $\psi(t)=\sum_{i=0}^{1}\Big(w^2(i,t)+w_t^2(i,t)+w_h^2(i,t)+\mu w_{ht}^2(i,t)\Big)$.
 	Then, an application of Gronwall's inequality to \eqref{dex1} shows
 	\begin{align}
 	e^{2\alpha t}\Bigg(\norm{\theta(t)}^2&+\mu\norm{\theta_x(t)}^2+\frac{1}{9}\frac{\mu}{\nu}E_1(\theta)(t)\Bigg)+\frac{\beta_1}{2}\int_{0}^{t}e^{2\alpha s}\Bigg(\norm{\theta_x(t)}^2+E_1(\theta)(s)\Bigg)\;ds\notag\\
 	&\qquad+\int_{0}^{t}e^{2\alpha s}\Big(\sum_{i=0}^{1}\frac{4}{3c_i}w_h^2(i,s)\theta^2(i,s)\Big) ds\\
 &	\leq Ch^4\int_{0}^{t}e^{2\alpha s}\Bigg((1+\mu)\norm{w(t)}^2_2+(1+\mu+\mu^2)\norm{w_t(s)}^2_2\Bigg)\; ds\notag\\
 	&\qquad\exp\Bigg(\int_{0}^{t}\Big(\mathbf{\phi}(s)+\mathbf{\psi}(s)+(w^4(0,s))^2+(w^4_h(0,s))^2\Big) ds\Bigg)\label{dex2}.
 	\end{align}
 	Multiplying \eqref{dex2} by $e^{-2\alpha t}$ and 
 	using Lemmas \ref{dlm1.2}, \ref{dlm1.4}, \ref{dlm2.2} and \ref{dlm2.4} with $\alpha=0,$ it follows that
 	\begin{align*}
 	\Bigg(\norm{\theta(t)}^2&+\mu\norm{\theta_x(t)}^2+\frac{1}{9}\frac{\mu}{\nu}E_1(\theta)(t)\Bigg)+\frac{\beta_1}{2}e^{-2\alpha t}\int_{0}^{t}e^{2\alpha s}\Bigg(\norm{\theta_x(t)}^2+E_1(\theta)(s)\Bigg)\;ds\\
 	&+e^{-2\alpha t}\int_{0}^{t}e^{2\alpha s}\Big(\sum_{i=0}^{1}\frac{4}{3c_i}w_h^2(i,s)\theta^2(i,s)\Big) ds\leq C\frac{1}{\mu}(\norm{w_0}_2)h^4e^{-2\alpha t} \exp\Big(C\norm{w_0}_2\Big).
 	\end{align*}
 	This completes the proof.
 \end{proof}
 \begin{lemma}\label{dlm3.3}
 	Assume that $w_0\in H^2(0,1)$. Then, there exists a positive constant C independent of h such that
 	\begin{align*}
 	\nu\norm{\theta_x(t)}^2+&\frac{1}{3}E_2(\theta)(t)+2e^{-2\alpha t}\int_{0}^{t}e^{2\alpha s}\Big(\norm{\theta_t(t)}^2+\mu\norm{\theta_{xt}(t)}^2\Big)\; ds\\
 	&\quad+\frac{\mu}{\nu}e^{-2\alpha t}\int_{0}^{t}e^{2\alpha s}\Big((1+c_0+w_d)\theta_t^2(0,s)+(1+c_1+w_d)\theta_t^2(1,s)\Big)\; ds\\
 	&\qquad \leq C\frac{1}{\mu}(\norm{w_0}_2)(1+\mu)h^4e^{-2\alpha t}\exp\big(C(\norm{w_0}_2)\big).
 	\end{align*}
 \end{lemma}
 \begin{proof}
 	Set $\chi=\theta_t$ in \eqref{deq4.7} to obtain
 	\begin{align}
 	\norm{\theta_t(t)}^2+&\mu\norm{\theta_{xt}(t)}^2+\frac{1}{2}\frac{d}{dt}\Bigg(\nu\norm{\theta_x(t)}^2+\sum_{i=0}^{1}(1+c_i+w_d)\theta^2(i,t)\Bigg)\notag\\
 	&\qquad+\frac{\mu}{\nu}\Big(\sum_{i=0}^{1}(1+c_i+w_d)\theta_t^2(i,t)\Big)\notag\\
 	&\quad =\sum_{{i=1}}^{5}I_i(\theta_t)\label{deq4.10}.
 	\end{align}
 	The first term on the right hand side of \eqref{deq4.10} is bounded by
 	\begin{align*}
 	I_1(\theta_t)=(\eta_t,\theta_t)-\mu\lambda(\eta_t,\theta_t)+\nu\lambda(\eta,\theta_t)&\leq \frac{1}{4}\norm{\theta_t(t)}^2+C(1+\mu^2)\norm{\eta_t(t)}^2+C\norm{\eta(t)}^2.
 	\end{align*}
 	For the second term $I_2(\theta_t),$ first rewrite it as
 	\begin{align*}
 	I_2(\theta_t)&=(1+w_d)(\eta_x-\theta_x,\theta_t)\\
 	&=-(1+w_d)\frac{d}{dt}(\eta,\theta_x)+(1+w_d)(\eta_t,\theta_x)+(1+w_d)\frac{d}{dt}\Big(\eta(1,t)\theta(1,t)-\eta(0,t)\theta(0,t)\Big)\\
 	&\qquad-(1+w_d)\big(\eta_t(1,t)\theta(1,t)-\eta_t(0,t)\theta(0,t)\big)-(1+w_d)(\theta_x,\theta_t).
 	\end{align*}
 	A use of Young's inequality shows
 	\begin{align*}
 	I_2(\theta_t)&=	(1+w_d)(\eta_x-\theta_x,\theta_t)\\
 	&\leq(1+w_d)\frac{d}{dt}\Big(\eta(1,t)\theta(1,t)-\eta(0,t)\theta(0,t)\Big)-(1+w_d)\frac{d}{dt}(\eta,\theta_x)\\ &\qquad+C\Big(\norm{\eta_t(t)}^2+\tnorm{\theta(t)}^2+\eta_t^2(1,t)+\eta_t^2(0,t)\Big)+\frac{1}{8}\norm{\theta_t(t)}^2.
 	\end{align*}
 	For	the third term $I_3(\theta_t)$ on the right hand side of \eqref{deq4.10}, we first rewrite it as
 	\begin{align*}
 	I_3(\theta_t)&=	\big(w(\eta_x-\theta_x)+(\eta-\theta)w_{hx},\theta_t\big)\\
 	&=-(w\eta,\theta_{xt})+\big(w(1,t)\eta(1,t)\theta_t(1,t)-w(0,t)\eta(0,t)\theta_t(0,t)\big)-(w_x\eta,\theta_t)\\
 	&\qquad-(w\theta_x,\theta_t)+\big((\eta-\theta)w_{hx},\theta_t\big)\\
 	&=-\frac{d}{dt}\big((w\eta),\theta_x\big)+\frac{d}{dt}\Big(w(1,t)\eta(1,t)\theta(1,t)-w(0,t)\eta(0,t)\theta(0,t)\Big)\\
 	&\qquad+\big((w\eta)_t,\theta_x\big)-\big(w_x\eta,\theta_t\big)-\Big(\big(w(1,t)\eta(1,t)\big)_t\theta(1,t)-\big(w(0,t)\eta(0,t)\big)_t\theta(0,t)\Big)\\
 	&\qquad-(w\theta_x,\theta_t)+\big((\eta-\theta)w_{hx},\theta_t\big).
 	\end{align*}
 	and then an application of Young's inequality with Lemmas \ref{dlm1.2}, \ref{dlm1.4} and \ref{dlm2.2} yields
 	\begin{align*}
 	I_3(\theta_t)=	\big(w(\eta_x-\theta_x)&+(\eta-\theta)w_{hx},\theta_t\big)\\
 	& \leq -\frac{d}{dt}\big((w\eta),\theta_x\big)+\frac{d}{dt}\Big(w(1,t)\eta(1,t)\theta(1,t)-w(0,t)\eta(0,t)\theta(0,t)\Big)\\
 	&\quad +C\Big(\norm{\eta_t(t)}^2+\norm{\eta(t)}^2+\eta^2(1,t)+\eta^2(0,t)+\eta_t^2(0,t)+\eta_t^2(1,t)
 	+\tnorm{\theta(t)}^2\Big)\\
 	&\quad+C\norm{\theta(t)}^2\Big(w_h^2(0,t)+w_h^4(0,t)+\norm{\Delta_h w_h(t)}^2\Big)+\frac{1}{8}\norm{\theta_t(t)}^2.
 	\end{align*}
 	The first subterm of the fourth term $I_4(\theta_t)$ on the right hand side of \eqref{deq4.10} can be rewritten for $i=0,\hspace{0.1cm}1$ as
 	\begin{align*}
 	(1+c_i+w_d)\eta(i,t)\theta_t(i,t)&=(1+c_i+w_d)\frac{d}{dt}\big(\eta(i,t)\theta(i,t)\big)-(1+c_i+w_d)\eta_t(i,t)\theta(i,t).
 	\end{align*}
 	Hence, we obtain
 	\begin{align*}
 (1+c_0+w_d)&\eta(0,t)\theta_t(0,t)+(1+c_1+w_d)\eta(1,t)\theta_t(1,t)\\
 	&\leq(1+c_0+w_d)\frac{d}{dt}\big(\eta(0,t)\theta(0,t)\big)+(1+c_1+w_d)\frac{d}{dt}\big(\eta(1,t)\theta(1,t)\big)\\
 	&\qquad+ C\Big(\eta_t^2(0,t)+\eta_t^2(1,t)+\theta^2(0,t)+\theta^2(1,t)\Big).
 	\end{align*}
 	For second subterm of the fourth  term on the right hand side, we note that for $i=0,\hspace{0.1cm}1$
 	\begin{align*}
 	\frac{2}{9c_i}&\Big(w^3(i,t)-w_h^3(i,t)\Big)\theta_t(i,t)\\
 	&\quad=-\frac{1}{18c_i}\frac{d}{dt}\theta^4(i,t)-\frac{1}{3c_i}\frac{d}{dt}\Big(w_h^2(i,t)\theta^2(i,t)\Big)+\frac{2}{9c_i}\eta^3(i,t)\theta_t(i,t)\\
 	&\qquad+\frac{2}{3c_i}\Big(w^2(i,t)\eta(i,t)-w(i,t)\eta^2(i,t)+w_h(i,t)\theta^2(i,t)\Big)\theta_t(i,t)\\
 	&\qquad+\frac{2}{3c_i}w_h(i,t)w_{ht}(i,t)\theta^2(i,t).
 	\end{align*}
 	Using Lemma \ref{dlm1.2}, it follows that for $i=0,\hspace{0.1cm}1$
 	\begin{align*}
 	\frac{2}{9c_i}\eta^3(i,t)\theta_t(i,t)-\frac{2}{9c_i}\frac{d}{dt}\big(\eta^3(i,t)\theta(i,t)\big)\leq C\eta^2(i,t)\theta^2(i,t)+C\eta^2(i,t)\eta_t^2(i,t),
 	\end{align*}
 	\begin{align*}
 	\frac{2}{3c_i}w^2(i,t)\eta(i,t)\theta_t(i,t)-\frac{2}{3c_i}\frac{d}{dt}\Big(w^2(i,t)\eta(i,t)\theta(i,t)\Big)\leq C\Big(\theta^2(i,t)+\eta^2(i,t)+\eta_t^2(i,t)\Big),
 	\end{align*}
 	and
 	\begin{align*}
 	-\frac{2}{3c_i}w(i,t)\eta^2(i,t)\theta_t(i,t)+\frac{2}{3c_i}\frac{d}{dt}\Big(w(i,t)\eta^2(i,t)\theta(i,t)\Big)\leq C\Big(\theta^2(i,t)+\eta^2(i,t)+\eta_t^2(i,t)\Big).
 	\end{align*}
 	Also, we obtain
 	\begin{align*}
 	\frac{2}{9c_i}3w_h(i,t)\theta^2(i,t)\theta_t(i,t)\leq \frac{2}{9c_i}\frac{d}{dt}\big(w_h(i,t)\theta^3(i,t)\big)+ C\big(w_{ht}^2(i,t)\theta^2(i,t)+\theta^4(i,t)\big).
 	\end{align*}
 	We note that
 	\begin{align*}
 	\frac{2}{3c_0}w_h(0,t)&w_{ht}(0,t)\theta^2(0,t)+\frac{2}{3c_1}w_h(1,t)w_{ht}(1,t)\theta^2(1,t)\\
 	&\leq C\Big(w_h^2(0,t)+w_{ht}^2(0,t)\Big)\theta^2(0,t)+C\theta^2(1,t)\Big(w_h^2(1,t)+w_{ht}^2(1,t)\Big).
 	\end{align*}
 	The first subterm of the fifth term $I_5(\theta_t)$ on the right hand side is bounded by
 	\begin{align*}
 	\frac{\mu}{\nu}&(1+c_0+w_d)\eta_t(0,t)\theta_t(0,t)+\frac{\mu}{\nu}(1+c_1+w_d)\eta_t(1,t)\theta_t(1,t)\\
 	&\leq \frac{\mu}{\nu}\Big(\frac{c_0}{10}+\frac{1}{2}(1+w_d)\Big)\theta_t^2(0,t)+\frac{\mu}{\nu}\Big(\frac{c_1}{10}+\frac{1}{2}(1+w_d)\Big)\theta_t^2(1,t)+\frac{\mu}{\nu}C\Big(\eta_t^2(0,t)+\eta_t^2(1,t)\Big),
 	\end{align*}
 	For the second subterm of the fifth term $I_5(\theta_t),$ we note that for $i=0,\hspace{0.1cm}1$
 	\begin{align*}
 	\frac{2}{9c_i}\frac{\mu}{\nu}\frac{d}{dt}\eta^3(i,t)\theta_t(i,t)\leq \frac{\mu}{\nu}\frac{c_i}{10}\theta^2_t(i,t)+C\frac{\mu}{\nu}\eta^4(i,t)\eta_t^2(i,t),
 	\end{align*}
 	\begin{align*}
 	-\frac{\mu}{\nu}\frac{2}{9c_i}\frac{\mu}{\nu}\frac{d}{dt}\eta^3(i,t)\theta_t(i,t)=-\frac{\mu}{\nu}\frac{2}{3c_i}\theta^2(i,t)\theta_t^2(i,t).
 	\end{align*}
 	Using Lemmas \ref{dlm1.2} and \ref{dlm3.1}, it follows that
 	\begin{align*}
 	\frac{2}{9c_i}\frac{\mu}{\nu}\frac{d}{dt}\Big(w^2(i,t)\eta(i,t)-w(i,t)\eta^2(i,t)\Big)\theta_t(i,t)\leq \frac{\mu}{\nu}\frac{c_i}{10}\theta_t^2(i,t)+C\frac{\mu}{\nu}\Big(\eta^2(i,t)+\eta_t^2(i,t)\Big).
 	\end{align*}
 	Also, it is valid using Young's inequality that
 	\begin{align*}
 	\frac{2}{9c_i}\frac{\mu}{\nu}\frac{d}{dt}\Big(3w_h(i,t)\theta^2(i,t)\Big)\theta_t(i,t)&\leq \frac{\mu}{\nu}\frac{c_i}{10}\theta_t^2(i,t)+\frac{2\mu}{3\nu c_i}\Big(w_h^2(i,t)\theta_t^2(i,t)+\theta^2(i,t)\theta_t^2(i,t)\Big)\\
 	&\quad+C\frac{\mu}{\nu}w_{ht}^2(i,t)\theta^4(i,t),
 	\end{align*}
 	\begin{align*}
 	-\frac{\mu}{\nu}\frac{2}{3c_i}\frac{d}{dt}\Big(w_h^2(i,t)\theta(i,t)\Big)\theta_t(i,t)=-\frac{2\mu}{3\nu c_i}w_h^2(i,t)\theta_t^2(i,t)-\frac{4\mu}{3\nu c_i}w_h(i,t)w_{ht}(i,t)\theta(i,t)\theta_t(i,t).
 	\end{align*}
 	Using Lemma \ref{dlm2.2}, it follows that
 	\begin{align*}
 	-\frac{4\mu}{3\nu c_0}&w_h(0,t)w_{ht}(0,t)\theta(0,t)\theta_t(0,t)-\frac{4\mu}{3\nu c_1}w_h(1,t)w_{ht}(1,t)\theta(1,t)\theta_t(1,t)\\
 	&\leq \frac{\mu}{\nu}\frac{c_0}{10}\theta_t^2(0,t)+C\frac{\mu}{\nu}w_{ht}^2(0,t)\theta^2(0,t)+\frac{\mu}{\nu}\frac{c_1}{10}\theta_t^2(1,t)+C\frac{\mu}{\nu}w_{ht}^2(1,t)\theta^2(1,t).
 	\end{align*}
 	Hence, from \eqref{deq4.10}, we obtain using Lemmas \ref{dlm1.4} and \ref{dlm3.1}
 	\begin{align*}
 	\big(\norm{\theta_t(t)}^2+&\mu\norm{\theta_{xt}(t)}^2\big)+\frac{d}{dt}\Big(\nu\norm{\theta_x(t)}^2+E_2(\theta)(t)+\sum_{i=0}^{1}\frac{2}{3c_i}w_h^2(i,t)\theta^2(i,t)\Big)\\
 	&\qquad+\frac{\mu}{\nu}\Big(\sum_{i=0}^{1}(1+c_i+w_d)\theta_t^2(i,t)\Big)\\
 	&\leq C\Big(\norm{\eta(t)}^2+(1+\mu)\norm{\eta_t(t)}^2\Big)+C(1+\mu)\Big(\norm{\theta_x(t)}^2+\sum_{i=0}^{1}\big(\eta^2(i,t)+\eta_t^2(i,t)\\
 	&\quad+w_{ht}^2(i,t)\theta^2(i,t)\big)\Big)+C\norm{\theta(t)}^2\norm{\Delta_hw_h(t)}^2+C(1+\mu)\big(\sum_{i=0}^{1}\theta^4(i,t)\big)\\
 	&\quad-2(1+w_d)\frac{d}{dt}(\eta,\theta_x)+2(1+w_d)\frac{d}{dt}\Big(\sum_{i=0}^{1}(-1)^{i+1}\eta(i,t)\theta(i,t)\Big)\\
 	&\quad-2\frac{d}{dt}\big((w\eta),\theta_x\big)+2\frac{d}{dt}\Big(\sum_{i=0}^{1}(-1)^{i+1}w(i,t)\eta(i,t)\theta(i,t)\Big)
 	\\
 	&\quad+2\sum_{i=0}^{1}(1+c_i+w_d)\frac{d}{dt}\big(\eta(i,t)\theta(i,t)\big)+\frac{d}{dt}\Big(\sum_{i=0}^{1}E_5(i,t)\Big),
 	\end{align*}
 	where
 	\begin{align*}
 	E_5(i,t)=\frac{4}{9c_i}\Big(\big(\eta^3(0,t)+w^2(0,t)\eta(0,t)-w(0,t)\eta^2(0,t)+w_h(0,t)\theta^2(0,t)\big)\theta(0,t)\Big)\quad i=0,1.
 	\end{align*}
 	Multiply the above inequality by $e^{2\alpha t}$ and use Lemmas \ref{dlm1.2}, \ref{dlm1.3}, \ref{dlm2.2} and \ref{dlm3.1} with bounds of nonlinear boundary terms as in Lemma \ref{dlm3.2} to arrive at
 	\begin{align*}
 	e^{2\alpha t}\big(\norm{\theta_t(t)}^2&+\mu\norm{\theta_{xt}(t)}^2\big)+\frac{d}{dt}\Bigg(e^{2\alpha t}\Big(\nu\norm{\theta_x(t)}^2+E_2(\theta)(t)+\sum_{i=0}^{1}\frac{2}{3c_i}w_h^2(i,t)\theta^2(i,t)\Big)\Bigg)\\
 	&+\frac{\mu}{\nu}e^{2\alpha t}\Big(\sum_{i=0}^{1}(1+c_i+w_d)\theta_t^2(i,t)\Big)\\
 	&\leq Ch^4e^{2\alpha t}(1+\mu)\Big(\big(\norm{w(t)}^2_2+\norm{w_t}^2_2\big)\Big)+C(1+\mu)e^{2\alpha t}\Big(\norm{\theta_x(t)}^2+\sum_{i=0}^{1}\big(w_{ht}^2(i,t)\theta^2(i,t)\\
 	&\quad+\theta^4(i,t)\big)\Big)+2c_0\frac{d}{dt}\Big(e^{2\alpha t}\eta(0,t)\theta(0,t)\Big)+2(2+c_1+2w_d)\frac{d}{dt}\Big(e^{2\alpha t}\eta(1,t)\theta(1,t)\Big)\\
 	&\quad-2\frac{d}{dt}\Big(e^{2\alpha t}\big((w+1+w_d)\eta,\theta_x\big)\Big)+2\frac{d}{dt}\Big(e^{2\alpha t}\big(\sum_{i=0}^{1}(-1)^{i+1}w(i,t)\eta(i,t)\theta(i,t)\big)\Big)\\
 	&\quad+Ce^{2\alpha t}\norm{\theta(t)}^2\norm{\Delta_h w_h(t)}^2+\frac{d}{dt}\Big(e^{2\alpha t}\big(E_5(0,t)+E_5(1,t)\big)\Big).
 	\end{align*}
 	Integrate from $0$ to $t$ and then multiply the resulting inequality by $e^{-2\alpha t}$ to obtain
 	\begin{align}
 	\Big(\nu\norm{\theta_x(t)}^2&+E_2(\theta)(t)+\sum_{i=0}^{1}\frac{2}{3c_i}w_h^2(i,t)\theta^2(i,t)\Big)+e^{-2\alpha t}\int_{0}^{t}e^{2\alpha s}\Big(\norm{\theta_t(t)}^2+\mu\norm{\theta_{xt}(t)}^2\Big)\; ds\notag\\
 	&\quad+\frac{\mu}{\nu}e^{-2\alpha t}\int_{0}^{t}e^{2\alpha s}\Big(\sum_{i=0}^{1}(1+c_i+w_d)\theta_t^2(i,s)\Big)\; ds \notag\\
 	&\leq C(1+\mu)h^4e^{-2\alpha t}\int_{0}^{t}e^{2\alpha s}\Big(\big(\norm{w(t)}^2_2+\norm{w_t(t)}^2_2\big)\Big)\; ds\notag\\
 	&+Ce^{-2\alpha t}\int_{0}^{t}e^{2\alpha s}\norm{\theta(t)}^2\norm{\Delta_hw_h(t)}^2 ds+C(1+\mu) e^{-2\alpha t}\int_{0}^{t}e^{2\alpha s}\Big(\norm{\theta_x(t)}^2\notag\\
 	&+w_{ht}^2(0,s)\theta^2(0,s)+w_{ht}^2(1,s)\theta^2(1,s)+\theta^4(0,s)+\theta^4(1,s)\Big) ds\notag\\
 	&\quad+\Big(\big((2c_0-2w(0,t))\big)\eta(0,t)\theta(0,t)+\big((4+2c_1+4w_d+2w(1,t))\big)\notag\notag\\
 	&\quad\eta(1,t)\theta(1,t)\Big)-2\Big((w+1+w_d)\eta,\theta_x\Big)+E_5(0,t)+E_5(1,t)\label{dex3}.
 	\end{align}
 	
 	Use Young's inequality and Lemma \ref{dlm1.2} to obtain
 	\begin{equation*}
 	-2\Big((w+1+w_d)\eta,\theta_x\Big)\leq \frac{\nu}{2}\norm{\theta_x(t)}^2+C\norm{\eta(t)}^2.
 	\end{equation*}
 	Again using Young's inequality and Lemma \ref{dlm1.2}, we arrive at
 	\begin{align*}
 	\Big((2c_0-2w(0,t))&\eta(0,t)\theta(0,t)+(4+2c_1+4w_d+2w(1,t))\eta(1,t)\theta(1,t)\Big)\\
 	&\leq \frac{c_0}{4}\theta^2(0,t)+\frac{(c_1+2(1+w_d))}{4}\theta^2(1,t)
 	+C\big(\eta^2(0,t)+\eta^2(1,t)\big).
 	\end{align*}
 	Bounding in a similar fashion as in Lemma \ref{dlm3.2}, we obtain a bound for the nonlinear boundary terms as follows
 	\begin{align*}
 	E_5(i,t)\leq \frac{2}{3c_i}w_h^2(i,t)\theta^2(i,t)+\frac{1}{9c_i}\frac{5}{6}\theta^4(i,t)+\frac{c_i}{4}\theta^2(i,t)+C\eta^2(i,t)\quad i=0,1.
 	\end{align*}
 	Finally, apply Gr\"onwall's inequality to \eqref{dex3} to arrive using Lemmas \ref{dlm1.2}, \ref{dlm1.4}, \ref{dlm2.1}-\ref{dlm2.3} and \ref{dlm3.2} at
 	\begin{align*}
 	\nu&\norm{\theta_x(t)}^2+\sum_{i=0}^{1}\Big((c_i+1+w_d)\theta^2(i,t)+\frac{1}{27c_i}\theta^4(i,t)\\
 	&\quad+2e^{-2\alpha t}\int_{0}^{t}e^{2\alpha s}\Big(\norm{\theta_t(t)}^2+\mu\norm{\theta_{xt}(t)}^2\Big)\; ds+\frac{\mu}{\nu}e^{-2\alpha t}\int_{0}^{t}e^{2\alpha s}\Big(\sum_{i=0}^{1}(1+c_i+w_d)\theta_t^2(i,s)\Big)\; ds\\
 	&\qquad \leq C\frac{1}{\mu}(\norm{w_0}_2)(1+\mu)h^4e^{-2\alpha t}\exp\big(C(\norm{w_0}_2)\big).
 	\end{align*}
 	This completes the proof.
 \end{proof}
 \begin{remark}
 	As a consequence of Lemma \ref{dlm3.3}, we obtain superconvergence result for $\tnorm{\theta(t)}$ which depends on $\frac{1}{\sqrt \mu}$. However, for proving optimal estimate, only one modification may be made to compute $\int_{0}^{t}\norm{\eta_t(t)}^2ds\leq Ch^2\int_{0}^{t}\norm{w_{xt}(t)}^2 ds$.
 	Hence, we obtain
 	\begin{equation}\label{denx2}
 	\tnorm{\theta(t)}=O(h),
 	\end{equation}
 	which does not depend on $\frac{1}{\sqrt \mu}$.
 	Now using triangle inequality with Lemmas \ref{dlm3.2} and \ref{dlm3.3} and \eqref{denx2}, we obtain the following result.
 \end{remark}
 \begin{theorem}\label{dthm3.1}
 	Let $w_0\in H^2(0,1)$. Then, the following error estimates hold for the state and control variables
 	\begin{align}\label{deqn.5}
 	\norm{(w-w_h)(t)}^2_r= O\Big(\frac{1}{\sqrt\mu}h^{2-2r}e^{-\alpha t}\Big),
 	\end{align}
 	where $r=0,1$
 	and 
 	\begin{equation*}
 	\tnorm{(w-w_h)(t)}=O\Big(he^{-\alpha t}\Big).
 	\end{equation*}
 \end{theorem}
 \begin{proof}
 	The proof follows from Lemmas \ref{dlm1.4}, \ref{dlm3.2} and \ref{dlm3.3} with a use of triangle inequality and \eqref{deq4.5}. 
 \end{proof}
 \begin{theorem}\label{dthm3.2}
 	For $w_0\in H^2(0,1),$ there exists a constant $C>0$ such that
 	\begin{align}
 	\norm{(w-w_h)(t)}_{L^\infty}= O\Big(\frac{h^2}{\sqrt\mu}e^{-\alpha t}\Big)\label{deqn.6}
 	\end{align}
 	and
 	\begin{align}
 	|v_i(t)-v_{ih}(t)|&:=|K_i(w(i,t))-K_i(w_h(i,t))|=
 	O\Big(\frac{h^2}{\sqrt\mu}e^{-\alpha t}\Big)\label{deqn.7},
 	\end{align}
 	where $i=0,$ $1$.
 \end{theorem}
 \begin{proof}
 	From Lemma \ref{dlm3.3}, we obtain a superconvergence result for $\tnorm{\theta(t)}$. Using the Poincar\'e-Wirtinger's inequality, it follows that
 	\begin{equation*}
 	\norm{\theta(t)}_{{L^\infty}(I)}\leq C\tnorm{\theta(t)}.
 	\end{equation*}
 	Now a use of triangle inequality with estimates of $\norm{\eta(t)}_{L^\infty}$ and$\norm{\theta(t)}_{L^\infty},$ we arrive at the estimate \eqref{deqn.6}.
 	To find \eqref{deqn.7}, we note that 
 	the error in the control law is given by
 	\begin{align*}
 	|v_0(t)-v_{0h}(t)|&:=|K_0(w(0,t))-K_0(w_h(0,t))|\\
 	&=|\frac{1}{\nu}\Big((1+c_0+w_d)(\eta(0,t)-\theta(0,t))+\frac{2}{9c_0}(w^3(0,t)-\tilde w_h^3(0,t))\\
 	&\quad -\frac{2}{9c_0}(w_h^3(0,t)-\tilde w_h^3(0,t))\Big)|\\
 	&\leq C\big(|\eta(0,t)|+|\theta(0,t)|\big)+\frac{C}{c_0}|\eta(0,t)|(w^2(0,t)+\eta^2(0,t))\\
 	&\quad+\frac{C}{c_0}|\theta(0,t)|(w_h^2(0,t)+\eta^2(0,t))\\
 	&\leq C\frac{h^2}{\sqrt \mu}\norm{w}_2\big(1+w^2(0,t)+\norm{w}^2_2\big)+C|\theta(0,t)|(1+w_h^2(0,t)+\norm{w}^2_2)\\
 	&\leq C\frac{h^2}{\sqrt \mu}e^{-\alpha t}\exp\big(C\norm{w_0}_2\big).
 	\end{align*}
 	Similarly, it follows that
 	\begin{align*}
 	|v_1(t)-v_{1h}(t)|&:=|K_1(w(1,t))-K_1(w_h(1,t))|\\
 	&\qquad\leq C\frac{h^2}{\sqrt \mu}e^{-\alpha t}\exp\big(C\norm{w_0}_2\big).
 	\end{align*}
 	This completes the proof.
 \end{proof}
 \section{Numerical experiments}
 In this section, we discuss the fully discrete finite element formulation of \eqref{deq1.7}  using backward Euler method with Neumann boundary control laws. 
 Here, the time variable is discretized by replacing the time derivative by difference quotient.
 Let $W^n$ be the approximation of $w(t)$ in $V_h$ at $t=t_n=nk.$ 
 Let 
 $0<k<1$ denote the time step size and  
 $t_n=nk,$ where $n$ is nonnegative integer. For smooth function $\phi$ defined on $[0,\infty),$
 set $\phi ^n=\phi(t_n)$ and $\bar{\partial}_t\phi^n=\frac{(\phi^n-\phi^{n-1})}{k}$.\\
 Using backward Euler method, the fully discrete scheme corresponding 
 $\{{W^n}\}_{n\geq 1}\in V_h$ is a solution of
 \begin{align}
 	(\bar{\partial}_tW^n,\varphi_h)&+\mu(\bar{\partial}_tW^n_x,\varphi_{hx})+\nu(W^n_x,\varphi_{hx})+(1+w_d)(W^n_x,\varphi_h)+(W^nW^n_x,\varphi_h)+\Big((c_0+w_d)W^n(0)\notag\\
 	&+\frac{2}{9c_0}(W^n(0))^3\Big)\varphi_h(0)+\Big((c_1+w_d)W^n(1)+\frac{2}{9c_1}(W^n(1))^3\Big)\varphi_h(1)\notag\\
 	&+\frac{\mu}{\nu}\Bigg(\Big((c_0+w_d)\bar{\partial}_tW^n(0)\varphi_h(0)+\frac{2}{9c_0}\bar{\partial}_t\big(W^n(0)\big)^3\varphi_h(0)\Big)+\Big((c_1+w_d)\bar{\partial}_tW^n(1)\notag\\
 	&\qquad+\frac{2}{9c_1}\bar{\partial}_t\big(W^n(1)\big)^3\Big)\varphi_h(1)\Bigg)=0 \quad \forall \varphi_h \in V_h\label{dex4.1}
 \end{align}
 with
 $W^0=w_{0h}.$
 At each time level $t_n$, the nonlinear algebraic system \eqref{dex4.1} is solved by Newton's method with initial guess $W^{n-1}$.
 For implicit scheme \eqref{dex4.1} in our case, CFL condition is not needed.  We take  time step $k=0.0001$ and mesh size $h=1/60$.
 \begin{example}
	Here, we have taken the initial guess (exact solution at $t=0$) $w_0=20(0.5-x)^3-3,$ where $3=w_d$ is a constant steady solution for the original problem. We do not know the exact solution $w(t)$. Choose $t=[0,3.5]$.
	We consider zero Neumann boundary condition, which is without control and mark it as uncontrolled solution. Then to check whether constant steady state solution $w_d=3$ is asymptotically stable, we take nonlinear Neumann boundary feedback controllers which are given in \eqref{deq1.8}-\eqref{deq1.9} for different values of $c_0$ and $c_1$ with $\mu=0.5$ and $\nu=0.5$.
 \end{example}
From the line denoted as 'uncontrolled soln' in Figure \ref{fig:d5.1}, we can clearly observe that $W^n$ does not go to zero, that is, constant steady state solution $w_d=3$ is not asymptotically stable with zero Neumann boundaries. We now observe that for various combination of $c_0$ and $c_1$, the discrete solution goes to zero exponentially, see Figure \ref{fig:d5.1}. Moreover from Figure \ref{fig:d5.1}, we can see that the optimal decay rate $\alpha$ (with $w_d=3$), $0<\alpha\leq \frac{1}{2}\min\Big\{\frac{\nu}{\mu+1}, \frac{\nu}{2\mu+\nu}, \frac{\nu(4+c_i)}{\nu+(4+c_i)\mu} (i=0,1)\Big\}$ happens when $c_0=1=c_1$, which verify our theoretical result in Lemma \ref{dlm1.1}. When $c_i(i=0,1)<1$, then decay rate for the state is slow compare to the case when $c_i(i=0,1)\geq 1$.
\begin{figure}
	\centering
	\includegraphics[height=7cm]{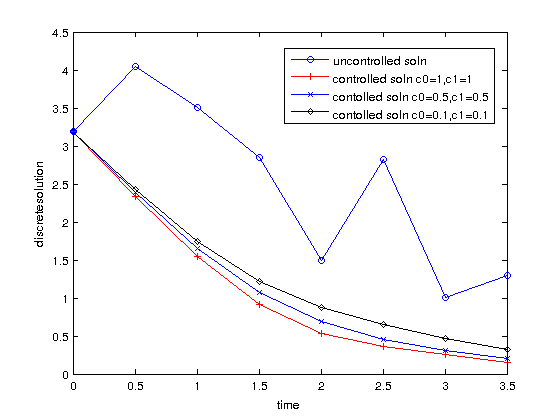}
	\caption{Both uncontrolled and controlled solution}
	\label{fig:d5.1}
\end{figure}

\begin{figure}[ht!]
	\begin{minipage}[b]{.5\linewidth}
		\centering
		
		\includegraphics[height=6cm]{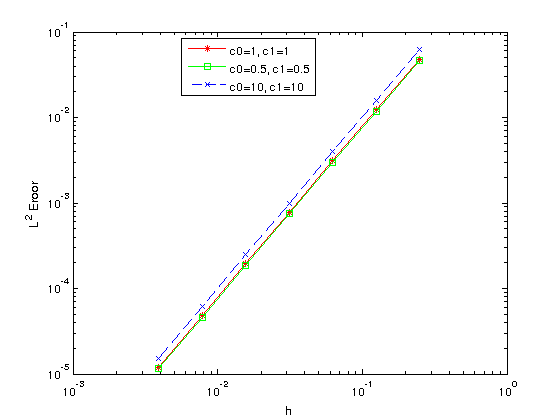}
		\caption{Order of convergence plot in $L^2$ norm}
		\label{fig:d5.2}
	\end{minipage}
	\hspace{0.05cm} 
	\begin{minipage}[b]{0.5\linewidth}
		\centering
		\includegraphics[height=6cm]{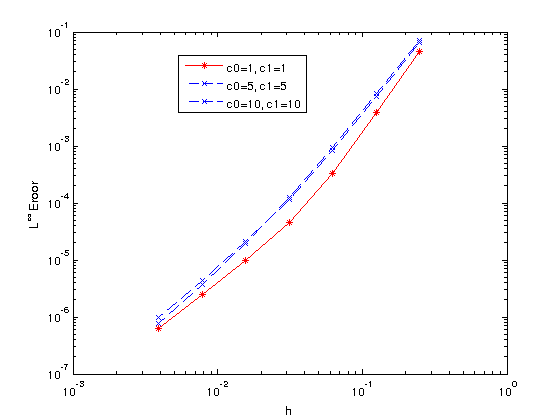}
		\caption{Order of convergence plot in $L^\infty$ norm}
		\label{fig:d5.3}
	\end{minipage}
\end{figure}

\begin{figure}[ht!]
	\begin{minipage}[b]{.5\linewidth}
		\centering
		
		\includegraphics[height=6cm]{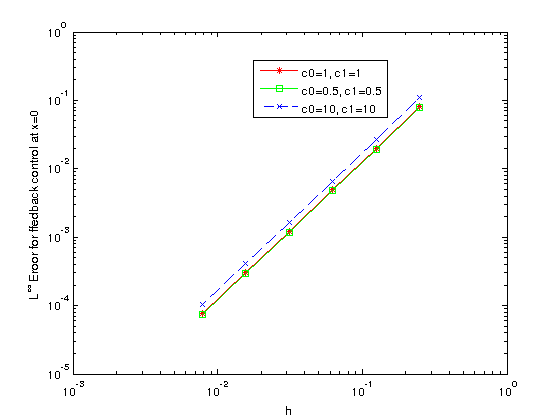}
		\caption{Convergence plot for feedback control error at $x=0$}
		\label{fig:d5.4}
	\end{minipage}
	\hspace{0.05cm} 
	\begin{minipage}[b]{0.5\linewidth}
		\centering
		\includegraphics[height=6cm]{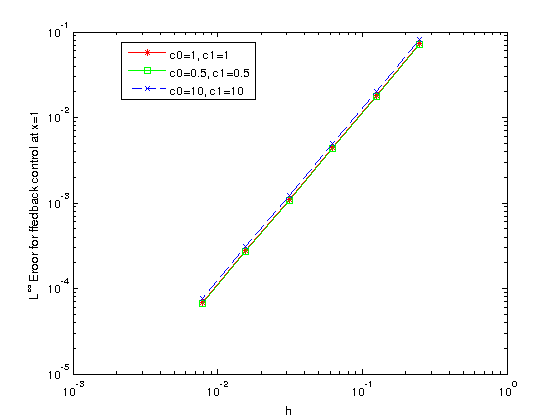}
		\caption{Convergence plot for feedback control error at $x=1$}
		\label{fig:d5.5}
	\end{minipage}
\end{figure}

\begin{figure}[ht!]
	\begin{minipage}[b]{.5\linewidth}
		\centering
		
		\includegraphics[height=6cm]{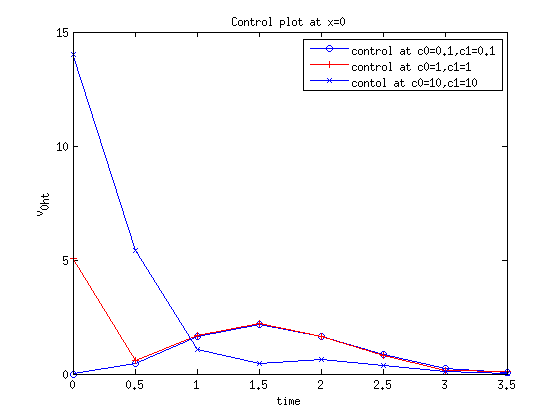}
		\caption{Control plot at $x=0$, namely; $v_{0ht}$}
		\label{fig:d5.6}
	\end{minipage}
	\hspace{0.05cm} 
	\begin{minipage}[b]{0.5\linewidth}
		\centering
		\includegraphics[height=6cm]{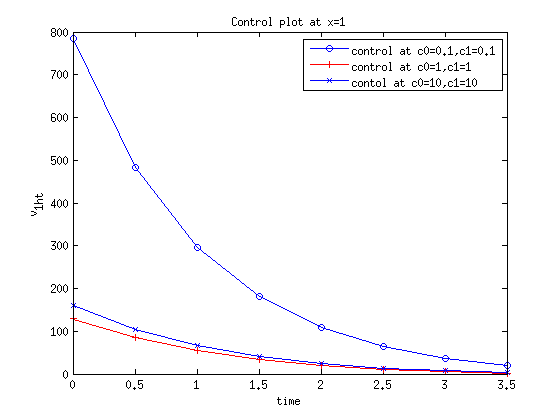}
		\caption{Control plot at $x=1$, namely; $v_{1ht}$}
		\label{fig:d5.7}
	\end{minipage}
\end{figure}
Now, we present order of convergence for the error in state variable $w(t)$ in $L^2$ and $L^\infty$ norms ($\norm{w(t_n)-W^N}_{L^2}$ and $\norm{w(t_n)-W^N}_{L^\infty}$ respectively) and also for the feedback controllers $v_0(t)$  and $v_1(t)$   ($|{v_0(t_n)-v_{0h}(t_n)}|$ and
$|{v_1(t_n)-v_{1h}(t_n)}|$ ) in $L^\infty$ norm at $t=1$. Exact solution is obtained through refined mesh solution.\\
Figures \ref{fig:d5.2} and \ref{fig:d5.3} indicate the error plot
for the state variable $w$ in $L^2$ and $L^\infty$ norms respectively, for various values of $c_0$ and $c_1$. We can easily observe from Figure \ref{fig:d5.2} that the convergence rate in the
$L^2$- norm for error in state variable is of order $2$ as predicted by Theorem \ref{dthm3.1}. From Figure \ref{fig:d5.3}, it is also noticeable that the order of convergence for error in state variable in $L^\infty$ norm is $2$ as expected from Theorem \ref{dthm3.2}.\\
For error in feedback controllers at $x=0$ and $x=1,$ it is observed from Figures \ref{fig:d5.4} and \ref{fig:d5.5} that for various values of $c_0$ and $c_1,$ the order of convergence is $2$ which confirms the result in Theorem \ref{dthm3.2}. 
In Figures \ref{fig:d5.6} and \ref{fig:d5.7}, we present the behavior of the feedback controllers at $x=0$ and $x=1$ with respect to time for various positive values of $c_0$ and $c_1$. Absolute value of the feedback controllers go to zero as time increases. So for $c_i(i=0,1)<1$ in the feedback control law, it will take more time for the control and state to settle down to zero (See Figures \ref{fig:d5.1}, \ref{fig:d5.6} and \ref{fig:d5.7}).\\
The next example consists of different type feedback control which is stated below.
In the following example, we consider the solution of \eqref{deq1.7} with one part zero Dirichlet boundary and another part different Neumann conditions.

\begin{example}\label{ex2}
In this example, we consider the solution of \eqref{deq1.7} with different boundary conditions. Take initial condition as $w_0=15\sin(\pi x)-5$, where $5$ is the steady state solution. We choose time $t=[0,10]$ and the time step $k=0.0001$ and $\mu=0.1$ and $\nu=0.1$.\\
For the uncontrolled solution, we take $w(0,t)=0$ and $w_x(1,t)=0$. The uncontrolled solution is denoted by 'uncontrolled soln'  in Figure \ref{fig:dx1}.\\
For the controlled solution we consider  $w(0,t)=0$ and $w_x(1,t)=v_1(t)=-\frac{1}{\nu}\Big((c_1+1+w_d)w(1,t)+\frac{2}{9c_1}w^3(1,t)\Big)$ with $c_1=1$ and $c_1=10$.
Denote the controlled solutions by 'controlled solution $c_1=1$',  'controlled solution with $c_1=10$', and 'controlled solution with $c_1=0.1$' in Figure \ref{fig:dx1}.
	\end{example}
\begin{figure}[ht!]
	\begin{minipage}[b]{.45\linewidth}
		\centering
		
		\includegraphics[height=6cm]{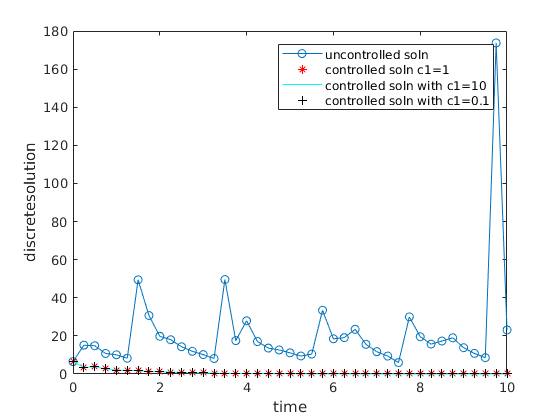}
		\caption{Controlled and Uncontrolled solution plot in $L^2$ norm}
		\label{fig:dx1}
	\end{minipage}
	\hspace{0.05cm} 
	\begin{minipage}[b]{0.45\linewidth}
		\centering
		\includegraphics[height=6cm]{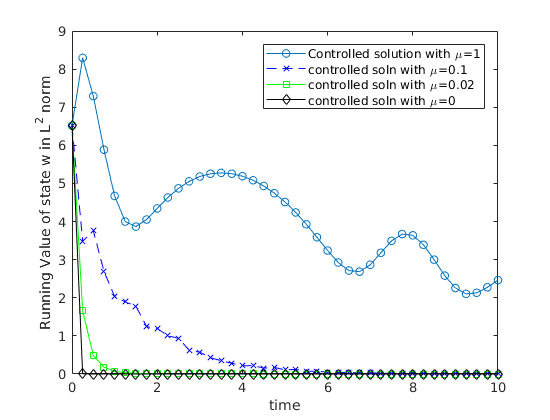}
		\caption{Decay of state $w$ in $L^2$ norm as $\mu\to 0$}
		\label{fig:dx2}
	\end{minipage}
\end{figure}
First draw line in Figure \ref{fig:dx1} shows that solution with zero boundary conditions ($w(0,t)=0$ and $w_x(1,t)=0$) oscillate.
But using above mentioned type of control with different values of $c_1$, solution  goes to zero.
With the initial condition of Example \ref{ex2}, decay of the state $w$ in $L^2$- norm varying $\mu$ with fixed $\nu=0.1$, $c_0=1=c_1$ is shown in Figure \ref{fig:dx2}. We observe that as $\mu$ decreases, $L^2$- norm of the state $w$ for BBM-B equation converges to the $L^2$- norm of the state $w$ with $\mu=0$ that is to the $L^2$- norm of the state of Burgers' equation.
\section{Conclusion}
In this article, under the assumption of the existence of solution, we show stabilization estimate in higher order norms
which is crucial to obtain optimal error estimates in the context of $C^0$- conforming  finite element analysis. Optimal error estimates for the state variable $w$ in $L^\infty(L^2)$,  $L^\infty(H^1)$ and $L^\infty(L^\infty)$ norms are established. Furthermore, superconvergence results for error in feedback controllers are also proved.
Following points which are itemized below will be addressed in a separate paper.
\begin{itemize}
	\item When the coefficient of viscosity is unknown (in the case of adaptive control), we believe that the control law as in Smaoui \cite{Smaoui} will also 
	work for BBM-B equation. Also when $\nu=0$, it is interesting to extend the analysis modifying the control law appropriately.
\item On the other hand, we have not discussed rigorously the existence of solution of problem  \eqref{deq1.7}-\eqref{deq1.10}, namely Theorem \ref{thmx1}.
\item In addition, for the fully discrete scheme \eqref{dex4.1}, it is interesting to know the large time behavior of the solution and how the corresponding time step size $k$ behaves in error estimates for fully discrete solution in addition to the space step size $h$.
\end{itemize}
{\bf{Acknowledgements.}}
The first author was supported by the ERC advanced grant 668998 (OCLOC) under the EUs H2020 research program. The first author  would like to thank Prof. Karl Kunisch for helpful suggestions. 
 %
 %
%
%
%
\bibliographystyle{amsplain}
 
\end{document}